\documentclass[11pt, reqno]{amsart}
\usepackage[margin=3.6cm]{geometry}
\usepackage{hyperref}
\usepackage{array,booktabs,tabularx}
\usepackage{graphicx}
\usepackage{amssymb}
\usepackage{enumerate}
\usepackage{color}
\usepackage{esint}
\usepackage{mathtools}
\usepackage{dsfont}
\usepackage{ esint }
\usepackage{enumitem}

\newtheorem{theorem}{Theorem}

\newtheorem{proposition}[theorem]{Proposition}

\newtheorem{Lem}[theorem]{Lemma}

\newtheoremstyle{named}{}{}{\itshape}{}{\bfseries}{.}{.5em}{\thmnote{#3 }#1} \theoremstyle{named} \newtheorem*{namedtheorem}{Theorem}

\theoremstyle{definition}
\newtheorem{definition}{Definition}
\newtheorem{remark}{Remark}

\newcommand{\tf}{\Rightarrow}
\newcommand{\leb}{\lambda}
\newcommand{\W}{\Omega}
\newcommand{\R}{{\mathbb R}}

\newcommand{\E}[1]{{\mathbb E}\left [#1\right]}

\newcommand{\ep}{\varepsilon}

\newcommand{\gives}{\ensuremath{\rightarrow}}
\newcommand{\x}{\ensuremath{\times}}

\newcommand{\abs}[1]{\ensuremath{\left| #1 \right|}}
\newcommand{\lr}[1]{\ensuremath{\left(#1 \right)}}
\newcommand{\norm}[1]{\left\lVert#1\right\rVert}
\newcommand{\inprod}[2]{\ensuremath{\left\langle#1,#2\right\rangle}}

\newcommand{\Union}{\ensuremath{\bigcup}}

\newcommand{\w}{\omega}
\newcommand{\dell}{\ensuremath{\partial}}
\newcommand{\set}[1]{\ensuremath{\{#1\}}}

\renewcommand{\a}{0}

\def\XXint#1#2#3{{\setbox0=\hbox{$#1{#2#3}{\int}$} \vcenter{\hbox{$#2#3$}}\kern-.5\wd0}}

\DeclareMathOperator{\vol}{vol}
\DeclareMathOperator{\Var}{Var}
\DeclareMathOperator{\diam}{diam}

\DeclareMathOperator{\inj}{inj}

\DeclareMathOperator{\Cov}{Cov}
\DeclareMathOperator{\Den}{Den}

\DeclareMathOperator{\MRW}{RW}

\DeclareMathOperator{\Crit}{Crit}
\DeclareMathOperator{\C}{{Crit}}

\title[Universality for Zeros and Crits of Random Waves]{Local Universality for Zeros and Critical Points of Monochromatic Random Waves}

\author[Y. Canzani]{Yaiza Canzani}
\author[B. Hanin]{Boris Hanin}

\address[Y. Canzani]{Department of Mathematics, University of North Carolina, Chapel Hill, United States.\medskip}
 \email{canzani@email.unc.edu}
\address[B. Hanin]{Department of Mathematics, Texas A\&M ~~and~~ Facebook AI Research, NYC\medskip}
\email{bhanin@math.tamu.edu}
\begin{document}
\maketitle
\begin{abstract}
This paper concerns the asymptotic behavior of zeros and critical points for monochromatic random waves $\phi_\leb$ of frequency $\leb$ on a compact, smooth, Riemannian manifold $(M,g)$ as $\leb\gives \infty.$ We prove global variance estimates for the measures of integration over the zeros and critical points of $\phi_\leb.$ These global estimates hold for a wide class of manifolds - for example when $(M,g)$ has no conjugate points - and rely on new local variance estimates on zeros and critical points of $\phi_\leb$ in balls of radius $\approx \leb^{-1}$ around a fixed point. Our local results hold under conditions about the structure of geodesics that are generic in the space of all metrics on $M.$ 
\end{abstract}

%==========================================================
%==========================================================
%==========================================================
%==========================================================

\section{Introduction}
This article gives new local and global results about the measures of integration over the zero and critical point sets of a monochromatic random wave $\phi_\leb$ of frequency $\leb$. To introduce our results, let $(M,g)$ be a compact, smooth, Riemannian manifold without boundary of dimension $n\geq 2,$ and write $\Delta_g$ for the positive definite Laplace-Beltrami operator. Consider an orthonormal basis $\set{\varphi_j}_{j=1}^\infty$ of $L^2(M,g)$ consisting of real-valued eigenfunctions
  $\Delta_g \varphi_j = \leb_j^2 \varphi_j$ {with} $0=\leb_0<\leb_1\leq \leb_2\leq \cdots \nearrow \infty$,  normalized so that $\norm{\varphi_j}_{2}=1$.
\emph{Monochromatic random waves} of frequency $\lambda$ are Gaussian fields on $M$ defined by
\begin{equation}
\phi_\leb:=\tfrac{1}{\sqrt{\dim(H_{\leb})}}\sum_{\leb_j\in [\leb,  \leb+1]} a_j \varphi_j,\label{E:RWdef}
\end{equation}
where the coefficients $a_j\sim N(0,1)$ are real valued  i.i.d standard Gaussians and 
\[H_{\leb}:=~\bigoplus_{\leb_j\in [\leb, \leb+1]}\ker(\Delta_g-\leb_j^2).\]
We write \[\phi_\leb\in \MRW_\leb(M,g)\] for short. The ensembles $\phi_\leb$ are Gaussian models for eigenfunctions of the Laplacian with eigenvalue approximately equal to $\leb^2$ on a compact Riemannian manifold $(M,g)$. In the setting of a general smooth manifold the ensembles $\MRW_\leb$ were first defined by Zelditch in \cite{Zel2}. Zelditch was inspired in part by the influential work of Berry \cite{Berry}, which proposes that random waves on Euclidean space and flat tori are good semiclassical models for high frequency wavefunctions in quantum systems whose classical dynamics are chaotic. 

{
 Specifically, Berry proposed his Random Wave Conjecture: at a fixed $x\in M$ the local behavior of \textit{deterministic} eigenfunctions ${\varphi_j}(x+u/\leb)$ should be well-approximated, as $\leb\gives \infty,$ by the behavior of frequency $1$  \textit{random} waves on $T_xM\cong \R^n$, so long as geodesics on $(M,g)$ are chaotic. The idea is that the rescaled eigenfunctions ${\varphi_j}\lr{x+u/\lambda}$ approximately solve $\Delta f = f$ for the ``frozen'' constant coefficient Laplacian on $T_xM\cong \R^n$ and hence are almost frequency $1$ eigenfunctions of $\Delta$ on $\R^n$ (see \eqref{e:ev1}). Moreover,  since a classical free particle on $(M,g)$ moves along geodesics, in the long time (i.e. high energy) limit, it will become equidistributed on $M$ with respect to the volume form of $g$ {if} the geodesics on $(M,g)$ are chaotic. Hence, its quantum analogs, the eigenfunctions ${\varphi_j}$ as $\leb_j\gives \infty$, should also become delocalized and have no preferred points or directions on $M,$ suggesting they are  locally well approximated by Euclidean random waves. 
 
Berry's Random Wave Conjecture has remained out of reach for deterministic sequences of eigenfunctions. However, the authors' previous work \cite{CH, CH-derivatives} can be viewed as a weak form of the conjecture: the statement that, under some assumptions on the structure of geodesics on $(M,g),$ rescaled \textit{random} waves $\phi_\lambda\lr{x+u/\lambda}$ on $(M,g)$ have Berry's random waves on $\R^n$ as their local limits in the sense of $C^\infty$ convergence of covariance functions of Gaussian fields (see Definition \ref{D:IS} and Remark \ref{R:IS}). This statement is non-trivial since it implies that the global geometry and topology of $(M,g)$ does not affect the local behavior of high frequency random waves. 

The purpose of the present work is to derive  several concrete results about the nodal sets and critical points of random waves  from the kernel convergence in \cite{CH, CH-derivatives}. This point of view has been taken up in a number of articles (c.f. e.g. \cite{Nourdin} and {Section \ref{S:Prior}}). } Define the measures of integration over $\{\phi_\leb=0\}$ and $\{d\phi_\leb=0\} $ by
\[
Z_\leb(\psi):= \int_{\phi_\leb^{-1}(0)}\psi(x) d\mathcal
H^{n-1}(x)\qquad \text{and} \qquad \C_\leb(\psi):=\sum_{d\phi_\leb(x)=0}\psi(x),
\]
where $\psi:M\gives\R$ and $\mathcal H^{n-1}$ is the $(n-1)$-dimensional Hausdorff measure. Our main result gives asymptotics for the expected value and estimates for the variance of the linear statistics of $Z_\leb, \text{Crit}_\leb$ that are valid for generic Riemannian metrics on $M$. For our estimates about the means of $Z_\leb(\psi),\C_\leb(\psi)$ we ask that $(M,g)$ be a manifold of isotropic scaling (see Definition \ref{D:IS} and  Section \ref{S:Spectral Conditions}). In particular, it is true for any manifold with negative curvature, or with no conjugate points. For our more delicate variance estimates to hold, we ask in addition that the restrictions of $\phi_\leb$ to small balls centered at different points become asymptotically uncorrelated. Namely, we say that the random waves $\phi_\leb$ have {\em short-range correlations} if for each $\ep>0$ and every $\alpha, \beta \in \mathbb N$
\begin{equation}\label{D:SRC}
%\sup_{\substack{x,y\in M\\ d_g(x,y) \geq \leb^{-1+\ep}}}
\sup_{\{x,y: \, d_g(x,y) \geq \leb^{-1+\ep}\}}
\abs{\nabla_x^\alpha\nabla_y^\beta
  \Pi_{\leb}(x,y)}=o_{\ep}(\leb^{\alpha + \beta}),
  \end{equation}
as $\leb \to \infty$, where $\nabla_x,\nabla_y$ are covariant derivatives, and $\Pi_\leb(x,y):=\Cov\lr{\phi_\leb(x),\phi_\leb(y)}$ is the two point correlation function for $\phi_\leb$, which is the kernel for the orthogonal projection onto $H_\leb.$ This condition is again generic in the space of Riemannian metrics on $(M,g)$ and is satisfied for example if for any pair of points $x,y\in M$ the measure of geodesic arcs joining them is zero. This is known to happen on manifolds of negative curvature, or more generally, with no conjugate points. We refer the reader to Section \ref{S:Spectral Conditions} for a discussion of when $(M,g)$ is a manifold with isotropic scaling satisfying \eqref{D:SRC}.

We are ready to sate our main theorem. It is the first result that gives variance estimates for  $Z_\leb, \text{Crit}_\leb$ that holds on a large class of smooth Riemannian manifolds (we refer the reader to Section \ref{S:Prior} for a discussion of previous work). In what follows $dv_g$ is the Riemannian volume form.

%%%%%%
\begin{theorem}\label{T:Global}
Let $(M,g)$ be a smooth, compact, Riemannian manifold of dimension
$n\geq 2$ with no boundary. Let $ \phi_\leb \in \MRW_\leb (M,g)$ and
suppose that $M$ is a manifold of isotropic scaling (Definition \ref{D:IS}). Then, for any bounded measurable function $\psi:M\gives \R$, 
\begin{equation}\label{E: expected size}
\lim_{\leb\gives \infty}\mathbb E \left[ \leb^{-1} Z_\leb(\psi)
\right] = \frac{1}{\sqrt{\pi n}} \frac{\Gamma \left(\frac{n+1}{2}
  \right)}{\Gamma \left(\frac{n}{2} \right)} \int_M \psi(x)
dv_g(x),
\end{equation}
and  
\begin{equation}\label{E: expected critical}
\lim_{\leb \gives \infty}\mathbb E \left[
  \leb^{-n}\Crit_\leb(\psi) \right]
= C_n\int_M \psi(x) dv_g(x),
\end{equation}
where $C_n$ is a positive constant that   depends only on $n$. Suppose further that $\phi_\leb$ has short-range correlations in the sense of \eqref{D:SRC}. Then,
\begin{equation}\label{E:var}
\Var \left[\leb^{-1} Z_\leb(\psi)\right] = O(\leb^{-\frac{n-1}{2}})
\end{equation}
and
\begin{equation}\label{E:var crits}
 \Var \left[ \leb^{-n}\Crit_\leb(\psi)\right]= O\lr{\leb^{-\frac{n-1}{2}}},
\end{equation}
as $\leb\gives \infty.$
\end{theorem}
%%%%%

\begin{remark}\label{R:jets}
The test function $\psi$ %in \eqref{E: expected size} and \eqref{E:var} 
can be replaced by a function 
  \[\psi(x)=\psi(x, \phi_\leb(x), D^2\phi_\leb(x),\ldots)\] 
depending on the jets of $\phi_\leb$ provided $\psi:\R^n\x C^0(\R^n, \R^k)\gives \R$ is bounded and continuous when $C^0(\R^n,\R^k)$ is equipped with the topology of uniform convergence on compact sets.
%{Similarly, the function $\psi$ in \eqref{E: expected critical} and \eqref{E:var crits} can be replaced by a bounded continuous function $\psi(x)=\psi(x, \phi_\leb(x), D^2\phi_\leb(x),\ldots)$ of the jets of $\phi_\leb.$} 
Hence, for  example, we could study the distribution of critical values by taking
\[\psi(u, \phi_\leb) = \mathbf 1_{\{\phi_\leb^x\geq \alpha\}}(u),\qquad \alpha  \in \R.\]

%\[\psi(u, \phi_\leb, D^2\phi_\leb) = \mathbf 1_{\{\phi_\leb^x\geq
%  \alpha,\,\,\text{sgn}(\text{Hess})(\phi_\leb^x)=q\}}(u),\qquad
%\alpha  \in \R, ~q\in \set{0,\ldots, n}.\]
\end{remark}
\begin{remark}
The proof of Theorem \ref{T:Global} actually shows that \eqref{E:
expected size} holds as soon as almost every point is a point of isotropic scaling. That is, it holds provided $\vol_g(M\backslash \mathcal{IS}(M,g))=0$ (see Definition \ref{D:IS}). 
  Also, by the Borel-Cantelli Lemma, if $n\geq 4$ and $\phi_j$ are independent frequency
  $j\in \mathbb N$ random waves on $(M,g)$, then \eqref{E:var} shows that the total
  nodal set measure $j^{-1}Z_{j}(\psi)-\E{j^{-1}Z_{j}(\psi)}$ converges
  almost surely to $0.$  Finally, when $n=2$ we have $C_2 = \E{\C_{\infty,1}}=\frac{1}{4\pi\sqrt{6}}$ where $C_2$ is the dimensional constant in \eqref{E: expected critical}.
\end{remark}

Theorem \ref{T:Global} hinges on a careful study of the statistics of $\phi_\leb$ when restricted to ``wavelength balls'' of radius $\approx \leb^{-1}$ around a fixed point $x\in M$ of isotropic scaling. The results that describe the behavior of $Z_\leb$ or $\text{Crit}_\leb$ restricted to these shrinking balls are given in Sections \ref{S:localzeros} and \ref{S:localcrits}, respectively. Before these two sections we  give an overview of the prior results in Section \ref{S:Prior} and give a formal definition of points of isotropic scaling in Section \ref{S:isotropic}.

%==========================================================
%==========================================================
%==========================================================
%==========================================================

\subsection{Prior results}\label{S:Prior}
To the best of our knowledge, Theorem \ref{T:Global} is the first result with a non-trivial variance estimate for the Hausdorff measure of the nodal set of random waves for a generic smooth Riemannian manifold (for real analytic $(M,g)$ a weaker estimate was given in \cite[Cor. 2]{Zel2}). A version of \eqref{E: expected size} was also stated, with a heuristic proof, in \cite[Prop. 2.3]{Zel2} for both Zoll and aperiodic manifolds.

Previous results on the Hausdorff measure of nodal sets focus primarily on exactly solvable examples, where more precise variance estimates are available. In these settings, due to the degeneracy of the spectrum of the Laplacian, one replaces random waves by random exact eigenfunctions. On round spheres, for instance, B\'erard \cite{Be2} proved \eqref{E: expected size} (example (1) on p.3). Later, in the same setting, Neuheisel \cite{Neu} and Wigman \cite{Wig2} obtained upper bounds for the variance that are of polynomial order in $\leb$. Further, on $S^2$, Wigman \cite{Wig} found that the variance actually grows like $\leb^{-2}\log \leb$ as $\leb \to \infty$, much better than the general $O(\leb^{-1/2})$ estimate in \eqref{E:var}. On flat tori $\mathbb T^n$ (for exact eigenfunctions) Rudnick and Wigman \cite{RW} computed the expected value of the total Hausdorff measure of the zero set and gave an upper bound of the form $\leb^2(\dim(H_{\leb}))^{-1/2}$ on its variance. Subsequently, on $\mathbb T^2$, Krishnapur, Kurlberg and Wigman \cite{KKW} found that the variance is asymptotic to a constant, while Marinucci, Pecatti, Rossi and Wigman proved that the size of the zero set converges to a limiting distribution that is not Gaussian and depends on the angular distribution of lattice points on circles \cite{MPRW}.

The behavior of the number of critical points has been studied in detail on $S^2$. Nicolaescu \cite{Nic} studied the expected value of the number of critical points, obtaining \eqref{E: expected critical}. The variance was studied by Cammarota, Marinucci and Wigman \cite{CMW}. They obtain a polynomial upper bound. This upper bound was later improved by Cammarota and Wigman \cite{CW} who proved that the variance grows like $\leb^2 \log \leb$ (as opposed to our $\leb^{7/2}$ estimate) as $\leb \to \infty$. Finally, for {a} smooth domain in $\R^2,$ Nourdin-Peccati-Rossi \cite{Nourdin} prove that both for real and complex random waves, the Hausdorff measure of the nodal set is asymptotically normal in the high frequency limit. 

%==========================================================
%==========================================================
%==========================================================
%==========================================================

\subsection{Isotropic scaling}\label{S:isotropic}

Our main result, Theorem \ref{T:Global}, hinges on understanding the statistics, as $\leb\gives \infty$, of $\phi_\leb$
restricted to ``wavelength balls'' of radius $\approx \leb^{-1}$ around
a fixed point $x\in M.$ After rescaling by $1/\leb$, the function $\phi_\leb$ has frequency approximately $1$ on such balls in the sense that it also the approximate local eigenvalue equation
\begin{equation}\label{e:ev1}
    \Delta_{_{\!T_xM}} \phi_\leb(x+\tfrac{u}{\leb})\approx \phi_\leb(x+ \tfrac{u}{\leb}),
    \end{equation}
where $\Delta_{_{\!T_xM}}$ denotes the \textit{flat} Laplacian on the tangent space at {$T_xM$}. Moreover, as $\leb\gives \infty,$ for a generic Riemannian metric on $M,$ the covariance kernels $\Pi_\leb$ of $\phi_\leb$ converge in the $C^\infty$ topology to those of a limiting ensemble of frequency $1$ functions $\phi_\infty$, called frequency $1$ random waves on $\R^n\cong T_xM$, where $n$ is the dimension of $M$. We explain this in what follows. 

It is natural to study $\phi_\leb$ by fixing $x\in M$ and considering the rescaled pullback of $\phi_\leb$ to the tangent space $T_{x}M.$ We denote this pullback by
\begin{equation}\label{E: rescaled}
\phi_\leb^{x}(u):=\phi_\leb\lr{\exp_{x}\lr{\frac{u}{\leb}}}.
\end{equation}
The law of $\phi_\leb\in \MRW_\leb(M,g),$ which is a centered smooth Gaussian field, is determined by its covariance kernel
\[\Pi_\leb(x,y):=\Cov\lr{\phi_\leb(x),\phi_\leb(y)}=\frac{1}{\dim H_{\leb}}\sum_{\leb_j\in [\leb, \leb+1]}\varphi_j(x)\varphi_j(y),\]
{$x,y \in M$.}
The function $\Pi_{\leb}(x,y)$ is the Schwartz kernel for the spectral (orthogonal) projector 
$\Pi_{[\leb, \leb+1]}:L^2(M,g)\gives H_{\leb},$
normalized to have unit trace. The dilated functions $\phi_\leb^x$ are centered Gaussian fields on $T_xM,$ and we denote their scaled covariance kernel by
\[ \Pi_{\leb}^{x}(u,v):=\Cov(\phi_\leb^{x}(u) ,\phi_\leb^{x}(v))=\Pi_{\leb}\lr{\exp_{x}\lr{\frac{u}{\leb}},~\exp_{x}\lr{\frac{v}{\leb}}},\]
{$u,v \in T_xM$.}
When $x$ is a point of isotropic scaling (see Definition \ref{D:IS}
below), the kernels $\Pi_\leb^x$ converge in the $C^\infty$
sense to the covariance kernel of a limiting ensemble of random
functions 
\begin{equation}\label{e:Berry}
\phi_\infty^x \in \MRW_1(T_{x}M, g_{x}),
\end{equation}
 called \textit{frequency $1$ random waves on $\R^n\cong T_xM$}. Here
 $g_x$ denotes the constant coefficient metric obtained by
 ``freezing'' $g$ at $x$. The random wave $\phi_\infty^x$ is defined as the unique centered Gaussian field with covariance
kernel
\begin{equation}\label{E:LimitDef}
\Pi_\infty^{x}(u,v)=\lr{2\pi}^{\frac{n}{2}}\frac{J_{\tfrac{n-2}{2}}\big(\abs{u-v}_{g_{x}}\big)} {\abs{u-v}_{g_{x}}^{\tfrac{n-2}{2}}}=\int_{S_{x}M} e^{i\inprod{u-v} {\w}_{g_{x}}}d\w.
\end{equation}
Here $J_\nu$ denotes a Bessel function of the first kind with index
$\nu,$ $S_{x}M$ is the unit sphere in $T_xM$ with respect to $g_{x},$ and $d\w$
is the hypersurface measure. The formal definition is the following. 

\begin{definition}\label{D:IS}
A point $x\in M$ is a \emph{point of isotropic scaling}, denoted $x
\in \mathcal {IS}(M,g),$ if for every non-negative function $r_\leb$ satisfying $ r_\leb=o(\leb)$ as $\leb\gives \infty$ and all $\alpha, \beta \in  \mathbb N^n$, we have
\begin{equation}\label{E:CovConv1}
\sup_{u, v \in B_{r_\leb}}  \left | \partial_u^\alpha \partial_v^\beta \,\left[\Pi_{\leb}^{x}(u,v)- \Pi_\infty^{x}(u,v) \right]\big. \right|=o_{\alpha,\beta}(1)
\end{equation}
as $\leb\gives \infty$, where the rate of convergence depends on $\alpha,
\beta$ and $B_R$ denotes a ball of radius $R$ centered at $0\in
T_xM.$ We also say that $M$ is a \textit{manifold of isotropic
  scaling} if $M=\mathcal{IS}(M,g)$ and if the convergence in
\eqref{E:CovConv1} is uniform over $x\in M$ for each $\alpha, \beta\in
\mathbb N^n.$
\end{definition} 

\begin{remark}\label{R:IS}
  If the set of geodesic loop directions $\mathcal L_{x,x} \subset S_x^*M$ through $x$ has measure $0,$ then $x\in \mathcal{IS}(M,g)$  by \cite[Thm. 1]{CH-derivatives}. This implication also holds if the spectral interval $[\leb, \leb+1]$ in the definition of $\phi_\leb$ is replaced by $[\leb, \leb+\eta(\leb)]$ with $\eta(\leb)=o(\leb)$ and $\liminf_{\leb \gives \infty}\eta(\leb) > 0.$ Even for these more general spectral windows, the condition that $M$ is a manifold of isotropic scaling is generic in the space of Riemannian metrics on  any smooth compact manifold. See \eqref{E:NSF} and Section \ref{S:Spectral Conditions} for details. { By \cite[Lem 6.1]{SZ}, the condition that $\abs{\mathcal L_{x,x}}=0$ for all $x\in M$ is generic in the space of Riemannian metrics on a fixed compact smooth manifold $M$, i.e. holds away from a countable union of nowhere dense sets.}
\end{remark}
\noindent { In addition, it is very likely that if $(M,g)$ has no conjugate points, then the condition
\[\lim_{\lambda\gives \infty}\log(\lambda)\cdot \eta_\lambda=\infty\] 
implies $\mathcal {IS}(M,g)=M.$  This was proved by B.~Keeler in \cite{Kee}, but with the convergence in \eqref{E:CovConv1} only holding for $\alpha=\beta=0$. This involves a non-trivial off-diagonal extension of B\'erard's estimates in \cite{Bon,Be}.}

If $x\in \mathcal{IS}(M,g)$, then in any coordinates around $x$ for which $g_{x}=\text{Id},$ the scaling limit of  waves in $\MRW_\leb(M, g)$ around $x$ is universal in the sense that it depends only on the dimension of $M.$ In the language of Nazarov-Sodin \cite{NS2} the asymptotics \eqref{E:CovConv1} imply that if $M=\mathcal{IS}(M,g),$ then the ensembles $\MRW_\leb(M,g)$ have translation invariant local limits. { For ensembles with such translation invariant local limits, Nazarov-Sodin \cite{NS2}, Sarnak-Wigman \cite{SW}, Gayet-Welschinger \cite{GW1,GW2,GW3} and Canzani-Sarnak \cite{CS}, as well as others, prove very interesting results on \textit{non-integral} statistics of the nodal sets of random waves. Such nodal set statistics include the number of connected components, Betti numbers, and topological types. A key step in all these articles, however, is to find some way to  reduce the study of non-integral statistics to integral statistics, such as the volume of the zero set or the number of critical points. Thus, the results on integral statistics in this article may be useful for future work on non-integral nodal set statistics of random waves as well.}

%==========================================================
%==========================================================
%==========================================================
%==========================================================
%==========================================================

\subsection{Local universality of zeros}\label{S:localzeros}
Our first result concerns the behavior of the nodal set of the
rescaled random wave $\phi_\leb^{x}$ for  $x\in \mathcal IS(M,g)$ (see
Definition \ref{D:IS}). Let us denote by $Z_\leb^{x}$ its Riemannian
hypersurface (i.e. Hausdorff) measure:
\[Z_\leb^x(A): = \mathcal H^{n-1}\lr{\lr{\phi_\leb^{x}}^{-1}(0)\,\cap\, A},\qquad \forall A\subseteq T_{x}M\text{ measurable}.\]
Theorem \ref{T:LocalZeros} concerns the restriction of $Z_{\leb}^{x}$
to various balls $B_r$ of radius $r$ centered at $0\in T_{x}M.$ We
set
\begin{equation}\label{E:NodalMeasDef}
Z_{\leb, r}^{x}:= \frac{\mathbf 1_{B_r}  \cdot 
  Z_\leb^x}{\text{vol}( B_r)} \qquad \text{and}\qquad Z_{\infty, r}^x:=\frac{\mathbf 1_{B_r}\cdot Z_\infty^{x}}{\text{vol}( B_r)}.
\end{equation}
We have denoted by $\mathbf 1_{B_r}$ the characteristic function of
the ball $B_r$ and by $Z_\infty^{x}$ the hypersurface measure on
$(\phi_\infty^{x})^{-1}(0)$ for $\phi_\infty^{x}\in\MRW_1(T_{x}M,
g_{x}).$ For various measures $\mu,$ we write $\mu(\psi)$ for
integration of a measurable function $\psi$ against $\mu.$ In particular, 
\[Z_{\leb, r}^{x}(1)= \frac{\mathcal H^{n-1}\lr{\lr{\phi_\leb^{x}}^{-1}(0)\,\cap\, B_r}}{\text{vol}( B_r) }.\]

\begin{theorem}[Weak Convergence of Zero Set Measures]\label{T:LocalZeros}
Let $(M,g)$ be a smooth, compact, Riemannian manifold of dimension
$n\geq 2$ with no boundary. Fix a non-negative function $r_\leb$ that satisfies $r_\leb=o(\leb)$ as $\leb\gives \infty$.  Let $ \phi_\leb \in \MRW_\leb (M,g)$ and $x \in \mathcal{IS}(M,g)$. Suppose  $\lim_{\leb\gives \infty}r_\leb $ exists and equals $r_\infty\in (0,\infty]$. \\

\noindent{\bf Case 1} $(r_\infty<\infty)\text{:}$ The measures
$Z_{\leb, r_\leb}^{x}$ converge to $Z_{\infty, r_\infty}^x$ weakly in
distribution. That is, for any bounded, measurable function $\psi:T_{x}M\gives \R$
\begin{equation}\label{E:NodalDist 1}
Z_{\leb, r_\leb}^x(\psi)\quad \stackrel{d}{\longrightarrow}
\quad Z_{\infty, r_\infty}^x(\psi)
\end{equation}
as $\leb \to \infty$, where $\stackrel{d}{\longrightarrow} $ denotes convergence in distribution.
 \medskip

\noindent{\bf Case 2} $(r_\infty=\infty)\text{:}$ We have the
following convergence in probability to a constant:
\begin{equation}\label{E:expectation}
Z_{\leb, r_\leb}^x(1) \quad \stackrel{p}{\longrightarrow}
\quad \frac{1}{\sqrt{\pi n}} \frac{\Gamma \left(\frac{n+1}{2} \right)}{\Gamma \left(\frac{n}{2} \right)},
\end{equation}
as $\leb\gives \infty.$ In particular, 
\begin{equation}\label{E: equal to infty}
\lim_{\leb\gives \infty}\Var\left[Z_{\leb, r_\leb}^x(1)\right]=0.
\end{equation}
\end{theorem}

After posting an early version of this article, G. Peccati brought to our attention that the convergence in distribution \eqref{E:NodalDist 1} for the nodal set measure $Z_{\lambda, r_\lambda}^x$ in balls of size $R/\lambda$ can be obtained directly from the $C^\infty$ scaling asymptotics of the covariance function. He kindly allowed us to include his argument, which we reproduce in the proof of Theorem \ref{T:LocalZeros} when $r_\infty<\infty$ (see Section  \ref{S:small-ball-local-zeros}).

\begin{remark}\label{R:zeros-jets1}
Just as in Remark \ref{R:jets}, the function $\psi$ in \eqref{E:NodalDist 1} can be allowed to depend on the jets $D^j\phi_\leb,\,j\geq 1.$ More precisely, $\psi(u)$ can be replaced by $\psi(u,W(u)),$ where $W$ is a random field so that $u\mapsto (\phi_\leb^x(u), W(u))$ is a continuous  Gaussian field with values in $\R^{1 + k}$ and $\psi:\R^n\x C^0(\R^n, \R^k)\gives \R$ is bounded and continuous when $C^0(\R^n, \R^k)$ is equipped with the topology of uniform convergence on compact sets. Since $(\phi_\leb^x(u), D\phi_\leb^x(u), D^2\phi_\leb^x(u),\ldots)$ is a smooth Gaussian field, we may take $W(u)=\lr{D^j\phi_\leb(u),\,j\geq 1}.$ Similarly, in \eqref{E:expectation} and \eqref{E: equal to infty}, the function $1=1(u)$ can be replaced by $\psi(W(u))$ where again $\psi:C^0(\R^n,\R^k)\gives \R$ is bounded and continuous in the topology of uniform convergence on compact sets. The only difference is that  \eqref{E:expectation} then reads
  \[Z_{\leb, r_\leb}^x(\psi) -\E{Z_{\infty,r_\leb}^x(\psi)} \quad \stackrel{p}{\longrightarrow}
\quad 0. \]
\end{remark}
\begin{remark}\label{R:generalsets} 
The relations \eqref{E:expectation} and \eqref{E: equal to
  infty} hold even if the balls
 $B_{r_\leb}$ in the definition of $Z_{\leb, r_\leb}^{x}$ are replaced by
 any $\leb-$dependent sets $A_{\leb,r_\leb}$ for which the diameter is bounded
 above and below by a constant times $r_\leb$, and whose volume tends
 to infinity when $r_\leb\gives \infty.$
\end{remark}
\begin{remark}\label{R:localunifzeros} 
The rates of convergence in \eqref{E:NodalDist 1}-\eqref{E: equal to
  infty} - even after the generalizations indicated in Remarks
\ref{R:zeros-jets1} and \ref{R:generalsets} - are uniform as $x $
varies over a compact set $ S\subset\mathcal{IS}(M,g)$ as long as
the convergence in \eqref{E:CovConv1} is uniform over $S.$
\end{remark}
\noindent Theorem \ref{T:LocalZeros} is proved in Section \ref{S:KR Zeros} and Section \ref{S: Limit theorem}. 

%=========================================================
%==========================================================
%==========================================================
%==========================================================
%==========================================================
%==========================================================
%==========================================================

\subsection{Local universality of critical points}\label{S:localcrits}
We state in this section our results on critical points of random
waves, which have been extensively studied
(c.f. e.g., \cite{CW,CMW,Nic}). Let $x\in M$ and for each
$r>0$ define the normalized counting measure
\begin{equation}
\C_{\leb,r}^x:=\frac{1}{\vol(B_r)}\sum_{\substack{d\phi_\leb^x(u)=0\\
  u \in B_r}} \delta_u \label{E:LocalCritMeas}
\end{equation}
of critical points in a ball of radius $r.$ We define
$\C_{\infty, r}^{x}$ in the same way as $\C_{\leb, r}^{x}$ but with $\phi_\leb^{x}$ replaced by $\phi_\infty^{x}\in \MRW_1(T_{x}M, g_{x}),$  and  continue to write
$\mu(\psi)$ for the pairing of a measure $\mu$ with a function  $\psi$. For example, 
\[\C_{\leb,r}^x(1)=\frac{\#\{ u \in B_r: \; d\phi_\leb^{x}(u)=0  \}}{\vol(B_r)}.\]

 \begin{theorem}\label{T:LocalCrits}
Let $(M,g)$ be a smooth, compact, Riemannian manifold of dimension
$n\geq 2$ with no boundary. Fix a non-negative function $r_\leb$ that
satisfies $r_\leb=o(\leb)$ as $\leb\gives \infty$. Let $ \phi_\leb \in
\MRW_\leb (M,g)$ and $x \in \mathcal{IS}(M,g)$. Suppose that
$\lim_{\leb\gives \infty}r_\leb$ exists and equals $r_\infty \in
(0,\infty].$ 
\begin{description}
\item[Case 1. ($r_\infty<\infty$)] For $k=1,2$ and
  each bounded measurable function $\psi:T_xM\to\R$ 
\begin{equation}\label{E:CritDist1}
\lim_{\leb \to \infty}\E{\C_{\leb, r_\leb}^{x}(\psi)^k}=
\E{\C_{\infty,r_\infty}^{x}(\psi)^k}.
\end{equation}
\item[Case 2. ($r_\infty=\infty$)] We have
\begin{equation}\label{E:CritDist2}
\lim_{\leb\gives \infty}\Var[\C_{\leb, r_\leb}^{x}(1)]=\E{\C_{\infty,1}^{x}(1)}.
\end{equation}
This limit is the expected
number of critical points in a ball of radius $1$ for frequency $1$
random waves on $\R^n,$ which is independent of $x.$
\end{description}
 \end{theorem}

 \begin{remark}\label{R:gensetscrits} 
   We prove in Section \ref{S:critsvariance} that the
   moments $\E{(\C_{\infty, r_\infty}^{x}(\psi))^k}$ are finite
   for $k=1,2.$ In particular, we show in Section \ref{S: Limit theorem}
   that if $\dim(M)=2,$ then $x \in M$
   \begin{equation}\label{E:crit loc}
   \E{  \C_{\infty,1}^x(1)}=\frac{1}{4\pi \sqrt{6}}.
   \end{equation}
Also, just as in Remark \ref{R:generalsets}, the balls
 $B_{r_\leb}$ in \eqref{E:CritDist2} can be replaced by
 any $\leb-$dependent sets $A_{\leb,r_\leb}$ for which the diameter is bounded
 above and below by a constant times $r_\leb$ and whose volume tends
 to infinity with $r_\leb.$
 \end{remark}
\begin{remark}\label{R:crits-jets1}
  Just as in Remark \ref{R:jets}, both  $\psi$ in \eqref{E:CritDist1} and the function $1$ being integrated against $\C_{\leb, r_\leb}^x$  in \eqref{E:CritDist2} can be replaced by a bounded continuous function of the jets of $\phi_\leb,$ giving information for instance about critical points filtered by critical value.  
\end{remark}
\begin{remark}\label{R:localunifcrits}
  Just as in Remark \ref{R:localunifzeros}, the rates of convergence
  in \eqref{E:CritDist1} and \eqref{E:CritDist2} - even after the
  generalizations indicated in Remaks \ref{R:gensetscrits} and
  \ref{R:crits-jets1} - are uniform over
  $x\in S\subset \mathcal{IS}(M,g)$ if \eqref{E:CovConv1} is uniform
  over $S.$
\end{remark}

On the $n$-dimensional flat torus, Nicolaescu \cite{Nic} obtained several results related to Theorem \ref{T:LocalCrits} in the $r_\infty<\infty$ case. We prove Theorem \ref{T:LocalCrits} in Section \ref{S:KR crits} and Section \ref{S: Limit theorem}.
% The statement \eqref{E:CritDist1} that we prove for critical points in
%Theorem \ref{T:LocalCrits} is weaker than the corresponding
%convergence in distbriution \eqref{E:NodalDist 1} for zeros in Theorem
%\ref{T:LocalZeros}. %The reason is that the random variables
%$Z_{\infty, r}(\psi)$ are bounded and hence determined by
%their moments. Indeed, a deterministic theorem of Donnelly-Fefferman
%\cite{DF} says that the Hausdorff measure of the zero set of $f\in
%\ker(\Delta_{\R^n}-1)$ is uniformly bounded when restricted to any
%fixed compact set. In contrast, although we do not have a proof of
%this fact, we believe that there exists a $k$ (depending only on $n$)
%so that $\mathbb E [\C_{\infty, 1}(1)^k ]=\infty.$ 

%==========================================================

\subsection{Sufficient conditions for isotropic scaling and short-range correlations}\label{S:Spectral Conditions}
To apply Theorems \ref{T:LocalZeros} and \ref{T:LocalCrits} one must verify that some $x\in M$ belongs to $\mathcal IS(M,g),$ the points of isotropic scaling (Definition \ref{D:IS}).  { This is typically difficult to do directly, but can be done by hand for eigenvalue windows of the form $[\lambda, \lambda+1]$ on simple examples, { such as flat tori}, by verifying a condition} about the geodesics through $x.$ Namely, for $x,y\in M$ denote by 
\begin{equation}\label{E:GeodArcs}
\mathcal L_{x,y}=\{\xi \in S_{x}M:\; \exists t>0 \text{ s.t. }\exp_{x}(t\xi )=y\}
\end{equation}
the set of directions that generate geodesic arcs from $x$ to $y.$ Here, $S_xM=\{\xi \in T_xM :\; |\xi|_{g(x)}=1\}$ is the unit sphere in $T_xM.$ Theorem 1 in
\cite{CH-derivatives} shows that 
\begin{equation}\label{E:NSF}
%x\in M\text{ is "non self-focal"}\quad \Longleftrightarrow \quad 
\abs{\mathcal L_{x,x}}=0 \quad \Longrightarrow \quad x\in \mathcal
IS(M,g),
\end{equation}
where $\abs{\mathcal L_{x,x}}$ denotes the volume of $\mathcal L_{x,x}$ inside
$S_xM.$ There is a similar sufficient condition for the short-range correlations assumption in Theorem
\ref{T:Global}: 
\begin{equation}\label{E:offdiagdecay}
\abs{\mathcal L_{x,y}}=0 \;\;\; \forall ~x,y\in M \quad \Longrightarrow
                          \quad RW_\leb(M,g)\text{ have short-range correlations}.
\end{equation}
Indeed, when $\abs{\mathcal L_{x,y}}=0$ for all $x,y\in M$ and any $\ep>0$ \cite[Thm. 3.3]{Saf} gives that for all $P,Q$ pseudodifferential operators on $M$ of with $\text{ord}P=\text{ord}Q$,
\begin{equation}
\sup_{x,y:\, d(x,y) \geq \ep} |P_x Q_y \Pi_{\leb}(x,y)|=o(\leb^{\text{ord}P+\text{ord}Q}),\label{E:Saf}
\end{equation}
as $\leb \to \infty$.  Here, the subscripts $x$ and $y$ indicate that
$P$ and $Q$ are acting on the $x$ and $y$ variables,
respectively. Note that $\Delta^{-\frac{m}{2}}
\Pi_{\leb}(x,y)=\leb^{-\frac{m}{2}} R  \Pi_{\leb}(x,y)$ with $R$ being an order zero pseudodifferential operator for any $m \in \mathbb N$. If we have that $\text{ord}P=\text{ord}Q+m$ with $m \in \mathbb N$, then 
\[P Q \Pi_{\leb}(x,y)= \leb^{-\frac{m}{2}} \,P Q \Delta^{\frac{m}{2}} R \Pi_{\leb}(x,y),\]
and the result follows since $Q \Delta^{\frac{m}{2}} R$ has the same
order as $P$. Combining \eqref{E:Saf} with Remark $3$ after Theorem
$2$ in \cite{CH-derivatives} yields \eqref{E:offdiagdecay}.

{ By \cite[Lem 6.1]{SZ}, the condition that $\abs{\mathcal L_{x,x}}=0$ for all $x\in M$ is generic in the space of Riemannian metrics on a fixed compact smooth manifold $M.$ {It is likely that}  a similar argument would show that $\abs{\mathcal L_{x,y}}=0$ for all $x,y\in M$ is also generic.  It is known, however, that {the condition that} $\abs{\mathcal L_{x,y}}=0$ holds for all $x,y\in M$ if  $(M,g)$ is negatively curved or, more generally, has no conjugate points.
}

%==========================================================
%==========================================================
%==========================================================
%==========================================================

\subsection{Novel aspects of the paper}
As mentioned above, this article gives what appear to be the first variance estimates for $Z_\leb, \text{Crit}_\leb$ that hold on a large class of smooth Riemannian manifolds. This is in contrast to the many articles mentioned in Section  \ref{S:Prior} that treat integrable models such as spheres and torii. The main new ingredients are the following. First, we use the spectral theory results in \cite{CH,CH-derivatives}, which show that the scaling limits of random waves at non-self focal points on $(M,g)$ are frequency $1$ random waves on $\R^n$ (see \eqref{E:CovConv1}). We do not use the spectral theory results in our proofs directly, but the fact that for a generic metric on $M$ all points are non self-focal means that the techniques in this article apply to generic $(M,g)$. Second, we give several new arguments about frequency $1$ random waves, which allow us to apply the Kac-Rice formula to zeros and critical points of frequency $1$ random waves (see around \eqref{E:separate} and Section  \ref{S: limiting ensemble} below). Finally, we give an essentially combinatorial argument for patching together our local variance estimates from Theorems \ref{T:LocalZeros} and \ref{T:LocalCrits} to obtain the global results in Theorem \ref{T:Global} (see Section  \ref{S: other cors}). 

Indeed, our method for studying zeros and critical points of random waves relies on the Kac-Rice formula. Many previous articles (e.g. \cite{CW, CMW, KKW, Neu, Nic, RW, Wig2, Wig, Zel2}) use the Kac-Rice formula to study the expected value and variance of the size of zero sets and number of critical points for random waves on flat tori and round spheres. In the vast majority of these cases, the Kac-Rice formula is not used directly. Instead, the authors explain that they cannot verify the non-degeneracy or the $1$-jet spanning hypotheses of the Kac-Rice Theorem (Theorem \ref{T:KR}). They then use modified, or approximate, Kac-Rice formulae adapted to each setting. In some instances this is unavoidable because the non-degeneracy hypothesis (2) can fail globally in the presence of many symmetries. 

We prove in contrast that the Kac-Rice formula can be applied to study \textit{all} moments of the zero and critical points sets of for frequency $1$ random waves $\phi_\infty$ on $\R^n.$ As we explained above, $\phi_\infty$ are the scaling limits of $\phi_\leb$ in balls of shrinking radii around a fixed $x\in M$ as long as $x$ is a point of isotropic scaling (see Definition \ref{D:IS} and Section \ref{S:KR Zeros} - Section \ref{S:KR crits}). To prove Theorems \ref{T:LocalZeros} and \ref{T:LocalCrits}, we use the Kac-Rice formula to write integral expression for moments of the measures of integration over the zero and critical points sets of the limiting random waves $\phi_\infty$ in wavelength balls around $x.$ For the first and second moments we are then able to check that we can apply the Kac-Rice formula directly for $\phi_\leb^x$ when $\leb$ is large but finite and that the $\leb\gives \infty$ limit of the resulting expressions are the original expressions for the limiting frequency one random waves on $T_xM$.

Proving that we can apply the Kac-Rice formula to study zeros and critical points of $\phi_\infty$ requires several new arguments that rely on analysis of frequency $1$ functions (i.e. smooth functions in $\ker(\Delta_{\R^n}-1)$) on $\R^n.$ For instance, combining Propositions \ref{P:Surjective zeros} and \ref{P:Surjective crits}, we find that frequency $1$
functions separate $1-$jets. More precisely, given $m$ distinct points $u_1,\ldots, u_m\in\R^n$ and constants $\set{\alpha_i, \beta_{i,j},\, \, 1\leq i \leq m,\, 1\leq j \leq n}$ there exists a smooth real-valued function $f\in \ker(\Delta_{\R^n}-1)$ such that
\begin{equation}\label{E:separate}
f(u_i)= \alpha_i\qquad \text{and}\qquad \dell_j f(u_i)=\beta_{i,j}.
\end{equation}
If $f$ were allowed to be any smooth function, then such a result is straightforward. However, with the restriction that $f$ have frequency precisely $1,$ we could not find such results in the literature. Let us also mention that after the writing of this article, it was explained to us by M. Sodin that there is a alternative complex analytic approach to proving such facts about frequency $1$ functions. 

This article also includes a new argument (see Section \ref{S: other cors}) for how to patch together local variance estimates for nodal and critical sets on balls of shrinking radii $\leb^{-\ep}$ to obtain quantitative upper bounds on the variance of the volumes of the zero and critical point sets of $\phi_\leb$ (see \eqref{E:var} and \eqref{E:var crits}). In ``integrating'' the local variance estimates, we control neither the rate at which the covariance kernels $\Pi_{\leb}$ of the random waves $\phi_\leb$ converge pointwise to their scaling limits near various $x\in M$ (see Definition \ref{D:IS}) nor the rate at which off-diagonal correlations decay (see Definition \ref{D:SRC}). Nonetheless, we are able to obtain quantitative variance estimates by using lower bounds on the volume of the set of points $(x,y)\in M\times M$ at which the spectral projection (covariance) kernel $\Pi_{\leb}(x,y)$ is already measurably small (of order $\leb^{-\frac{n-1}{2}}$, where $n$ is the dimension of $M$). This is the content of Section \ref{S: other cors} and is related in spirit to the work of Jakobson-Polterovich \cite{JP}. 
%==========================================================
\subsection{Acknowledgements} 
We are grateful to an anonymous referee for finding a gap in a
previous version of this article and for significantly improving some aspects of the exposition. That version concerned a wide class
of integral statistics (not just zeros and critical points) of
monochromatic random waves. However, there was an error in the
previous incarnation of what are now Propositions \ref{P:Surjective zeros}
and \ref{P:Surjective crits}. The new propositions fix the mistake for
the special cases of zeros and critical points. We leave the extension
of the results in this paper to more general integral statistics for
future work. The second author would also like to thank Damien Gayet
and Thomas Letendre for several useful discussions pertaining to the
arguments in Proposition \ref{P:L1Density}.

%==========================================================
\subsection{Outline} The rest of our paper is organized as
follows. First, in Section \ref{S: limiting ensemble}, we recall a variant
of the Kac-Rice formula and prove that it can be applied to study all
moments for the measures of integration over the zeros and critical
points of frequency $1$ random waves on $\R^n.$ We then complete the
proof of the our local results (Theorems \ref{T:LocalZeros} and
\ref{T:LocalCrits}) in Section \ref{S: Limit theorem}. Finally, in Section
\ref{S: other cors}, we explain how to use the assumption that random
waves have short-range correlations on $(M,g)$ (see Definition
\ref{D:SRC}) to prove our global results (Theorem \ref{T:Global}).

%==========================================================
%==========================================================
%==========================================================
%==========================================================
\section{Analysis of Frequency $1$ Random Waves on $\R^n$}\label{S: limiting ensemble}
Let $\phi_\infty$ be a frequency $1$
random wave on $\R^n$ (see \eqref{E:LimitDef}). We prove in this section that the following variant of the
Kac-Rice formula (an amalgam of Aza\"is-Wshebor \cite[Thms.\,6.2,\,6.3,\,\,Props.\,6.5,\,6.12]{AW}) can be applied to study all the moments for the measures of integration over its zero and critical point sets. In what follows we write $\mathcal H^k$ for the dimension $k$ Hausdorff measure on $\R^n.$

\begin{namedtheorem}[Kac-Rice]\label{T:KR}
Let $U$ be an open subset of $\R^n$ and $X:U \gives \R^{k}$ be a
Gaussian field with $k \leq n.$ Fix $m\in \mathbb N,$ and suppose that      \smallskip                  
  \begin{enumerate}
  \item $X$ is almost surely $C^2.$\smallskip
  \item {\bf Non-degeneracy:}  For every collection of distinct points $\set{u_j}_{j=1}^m$ the Gaussian vector $\lr{X(u_j)}_{j=1}^m$ has a non-degenerate distribution. 
  \end{enumerate}\smallskip
If $k<n$ suppose in addition that
\begin{enumerate}
 \item[(3)] {\bf 1-Jet Spanning Property:} For $u\in U,$ the joint
    distribution of $X(u)$ and the Jacobian $\lr{\dell_i
      X_j(u)}_{\substack{1\leq i \leq n \\ 1\leq j \leq k}}$ is a
    non-degenerate Gaussian.\smallskip
\end{enumerate}
Then, if $k<n$, 
\begin{equation}
\E{\mathcal H^{n-k}\!\lr {\set{X=\a}\cap B}^m}\!\!=\!\!\int_{B^m}\!\!\! Y_{_{\!m,X}}\lr{u_1,\ldots, u_m}\!\Den_{X(u_1),\ldots, X(u_m)}\!\!\lr{\a,\ldots, \a}\!du_1\dots du_m\label{E:KR}
\end{equation}
for every measurable Borel set $B\subseteq U,$ where
\begin{equation*}
 Y_{_{\!m,X}}(u_1,\ldots, u_m)=\E{\prod_{j=1}^m \left[\det\lr{dX(u_j)^* dX(u_j)}\right]^{1/2}\,\Big|\, X(u_j)=\a,\, j=1,\ldots, m}\label{E:KRDen}
\end{equation*}
and $\Den_{X(u_1),\ldots, X(u_m)}$ is the density of $\lr{X(u_j)}_{j=1}^m$. \\ \ \\
If $k=n,$ then equation \eqref{E:KR} holds for any Borel set
$B\subseteq U$ with the left hand side replaced by the factorial moment:
\[\E{\prod_{j=1}^m \left[\mathcal H^{n-k}\lr {\set{X=\a}\cap B}-j+1\right]}.\]
\end{namedtheorem}
\begin{remark}\label{R:KR Extension} 
The equality \eqref{E:KR} is valid even if one side of it (and hence
the other) is infinite. Moreover, let $W:\R^\alpha\gives \R^\beta$ be
a continuous Gaussian field such that   
$(X,W)$ is Gaussian and suppose $f:\R^n\x C^0(\R^\alpha,
\R^\beta)\gives \R$ is a positive measurable  
function that is continuous when $C^0(\R^\alpha, \R^\beta)$ is
equipped with with topology of uniform convergence on compact
sets. Then, the formula \eqref{E:KR} is valid with $\mathcal H^{n-k}\lr{\set{X=\a}\cap B}$ replaced by 
\[\int_{\set{X=\a}\cap B} f(u,W(u)) d\mathcal H^{n-k}|_{\{X=\a\}}(u),\] 
and $ Y_{_{\!m,X}}\lr{u_1,\ldots, u_m}$ replaced by $Y_{_{\!m,X,f}}\lr{u_1,\ldots,
  u_m},$ defined as
\begin{equation}\label{E:KRmod}
\E{\prod_{j=1}^m f\lr{u_j,\, W(u_j)}\left[\det\lr{dX(u_j)^* dX(u_j)}\right]^{1/2}\,\big|\, X(u_j)=\a,\, j=1,\ldots,m}.
\end{equation}
This statement when $f$ is bounded is a special case of
\cite[Thm. 6.10]{AW}. It can be extended to positive $f$ by
considering the truncations $f_N:=\max\lr{f,N},\, N\in \mathbb N,$ and using the monotone convergence theorem. 
\end{remark}

\subsection{Kac-Rice Hypotheses for Zeros}\label{S:KR Zeros} In this section, we prove that $\phi_\infty$ satisfies the hypotheses of the Kac-Rice Theorem. Since $\phi_\infty$ is almost surely smooth, Hypothesis (1) is satisfied. Hypothesis (2) requires that the distribution
of 
\[\text{ev}(\phi_\infty; u_1,\ldots, u_n): = \lr{\phi_\infty(u_1),\ldots, \phi_\infty(u_m)}\]
{be} non-degenerate. Note that { we have the following equality in law on the space of continuous functions:}
\begin{equation}\label{E:KLExpansion}
\phi_\infty(u) \stackrel{d}{=}\int_{S^{n-1}} \cos(\inprod{u}{\w})
\lr{\sum_{j=1}^\infty a_j \psi_j(\w)}d\w,
\end{equation}
where $\sum_{j=1}^\infty a_j \psi_j(\w)$ is a white noise based in
$L^2(S^{n-1},\R)$ (i.e. $a_j$ are i.i.d. standard Gaussians and $\set{\psi_j}$ is an
orthonormal basis for $L^2(S^{n-1}, \R)$). Let us write
\[V := \mathcal F^{-1}(L^2(S^{n-1},
\R))= \left\{ f(u)=\int_{S^{n-1}}\cos(\inprod{\w}{u})
  g(\w)d\w:\; \; g\in L^2(S^{n-1},\R) \right\}\]
for the real-valued functions on $\R^n$ with frequency $1.$ Since $\text{ev}\lr{\cdot\; ; u_1,\ldots, u_m}$ is
linear and the law of $\phi_\infty$ is a non-degenerate Gaussian
measure on $V$, it is enough to show
that $\text{ev}(\cdot\; ; u_1,\ldots, u_m)$ is surjective as a function
on $V$ for every fixed collection of $m\geq 1$ distinct points
$\set{u_\ell}_{\ell=1}^m.$ The surjectivity of the linear functional $\phi \mapsto
\text{ev}(\phi ; u_1,\ldots, u_m)$ is equivalent to the linear
independence of its components 
\begin{equation}\label{E:phi}
\phi\mapsto \phi(u_\ell)=
\int_{S^{n-1}}\cos(\inprod{u_\ell}{\w})\widehat{\phi}(\w) d\w. 
\end{equation}
Since $L^2(S^{n-1},\R)$ separates points, this is implied by taking
the real part of the following result. 

\begin{proposition}[Non-degeneracy for zero sets]\label{P:Surjective zeros}
Fix $n\geq 2$ and $m\geq 1.$ Let $u_1,\ldots ,u_m\in \R^n$ be distinct. Then, the functions
\[\{e^{i\inprod{u_\ell }{\w}}:\;1\leq \ell \leq m\}\]
are linearly independent on $S^{n-1}$. 
\end{proposition}

\begin{proof}
Suppose
 \begin{equation}\label{E:LD}
\sum_{\ell=1}^m a_\ell \; e^{i\inprod{u_\ell}{\w}} \equiv 0,\qquad a_\ell \in\R.
\end{equation}
By multiplying by $e^{-i\inprod{u}{\w}}$ for an appropriate $u\in \R^n$
we may assume that the values $\abs{u_\ell}$ are positive and distinct. Recall the plane wave expansion (see e.g. \cite[Thm. 2]{BDS})
\begin{equation}\label{E:PWE}
  e^{i\inprod{u}{\w}} = \sum_{k= 0}^\infty C_k \lr{\frac{i\,|u|}{2}}^k
  j_{k+\alpha}(|u|)Z_k(\widehat u, \w),
\end{equation}
where 
\[\alpha = \frac{n-2}{2},\quad C_k = c_k\cdot d_k,\quad c_k= \frac{\Gamma(\alpha
  +1)}{\Gamma(\alpha + k +1)},\quad \widehat{u}=\frac{u}{\abs{u}},\]
$d_k$ is the dimension of the space of spherical harmonics of degree $k$, 
the functions $j_\nu$ are normalized Bessel functions
\[j_\nu(t) = \Gamma\lr{\nu+1} \lr{\frac{t}{2}}^{-\nu} J_\nu(t)\]
solving $y'' + \frac{2\nu + 1}{t} y' + y =0$ with $y(0)=1$, and
$Z_k(\widehat{u}, \w)$ are the zonal harmonics of degree $k$
normalized by $\norm{Z_k(\widehat{u},\cdot)}_{L^\infty}=1$ for each
$\widehat{u},k.$ The normalization of $Z_k$ implies
\[\int_{S^{n-1}} Z_k(\w_0, \w)Z_k(\w_1, \w)d\w = d_k^{-1}Z_k(\w_0, \w_1),\]
for all $\w_0,\w_1 \in S^{n-1}$.
 Substituting
\eqref{E:PWE} into \eqref{E:LD}, we have
\[\sum_{\ell=1}^m \sum_{k=0}^{\infty} C_k  a_\ell\lr{\frac{i\abs{u_\ell}}{2}}^k j_{k+\alpha}(\abs{u_\ell})
Z_k\lr{\widehat{u}_\ell,\, \w} \equiv 0.\]
For each $\widehat{y}\in S^{n-1}$ and $k\geq 0$ we integrate against
$Z_{k}(\widehat{y}, \w)$ to find 
\begin{equation}\label{E:LD2}
\sum_{\ell=1}^ma_\ell c_k\lr{\frac{i\abs{u_\ell}}{2}}^{k} j_{k+\alpha}(\abs{u_\ell})
Z_{k}(\widehat{u_\ell}, \widehat{y})=0,\qquad \forall\;  k\geq 0,\,\,\;
\widehat{y}\in S^{n-1}.
\end{equation}
Let $\ell^* = \text{argmax}\set{\abs{u_\ell}~:~1\leq \ell \leq m},$
and recall that for $t\geq 0$ fixed 
\begin{equation}\label{E:BesselAsymptotics}
j_\nu(t) = 1 + o(1),\qquad \text{as  }\nu\gives \infty.
\end{equation}
Keeping in mind the normalization $\norm{Z_k(\widehat{u},\cdot)}_{L^\infty}=
Z_k(\widehat{u}, \widehat{u})=1$, we divide \eqref{E:LD2} by $c_k\lr{i
  \abs{u_{\ell^*}}/2}^k$, set $\widehat{y} = \widehat{u}_{\ell^*},$
use \eqref{E:BesselAsymptotics}, and send $k\gives \infty$ to conclude $a_{\ell^*}=0.$
Repeating this for the $m-1$ remaining points completes the proof. 
\end{proof}

It remains to check that $\phi_\infty$ satisfies Hypothesis (3) in Theorem \ref{T:KR}. Since the law of $\phi_\infty$ is translation-invariant, it is enough to show that $(\phi_\infty(0), \dell_1 \phi_\infty(0),\ldots, \dell_n \phi_\infty(0))$ is a non-degenerate Gaussian vector. Just as with the discussion before Proposition \ref{P:Surjective zeros}, but using the maps $ \phi \mapsto \partial_i \phi(0)=\int_{S^{n-1}}\w_i\, \widehat{\phi}(\w) d\w $, this is equivalent to the statement that the restrictions $\set{1, \w_1,\ldots, \w_n}$ of $1$ and the $n$ coordinate functions are linearly independent functions on $S^{n-1}.$ This is true since the zero set of an affine function is affine, and hence if it contains the unit sphere, then it must be $\R^n.$

\subsection{Kac-Rice Hypotheses for Critical Points}\label{S:KR crits}
We continue to write $\phi_\infty$ for a frequency $1$ random wave on $\R^n.$ The purpose of this section is to  check that the hypotheses of the Kac-Rice formula are satisfied by the Gaussian field $d\phi_\infty = \lr{\dell_1\phi_\infty,\ldots,  \dell_n\phi_\infty}.$ As in Section \ref{S:KR Zeros}, $d\phi_\infty$ is almost surely smooth and hence satisfies Hypothesis (1). Also, since $d\phi_\infty$ takes values in $\R^n$, we do not need to check Hypothesis (3) in the statement of the Kac-Rice theorem. It therefore remains to check Hypothesis (2). We must show that for any distinct $u_1,\ldots, u_m\in\R^n$ the vector $(\dell_i \phi_\infty(u_\ell))_{1\leq i \leq n,\,\,  1\leq \ell \leq m}$ has a non-degenerate distribution. By the same reasoning as preceded Proposition \ref{P:Surjective zeros}, but using the maps
\[\phi\mapsto  \partial_i\phi(u_\ell)=
\int_{S^{n-1}} \w_i\, \sin(\inprod{u_\ell}{\w})\widehat{\phi}(\w) d\w\]
  instead of \eqref{E:phi}, the non-degeneracy of $(\dell_i \phi_\infty(u_\ell))_{1\leq i \leq n,\,\,
  1\leq \ell \leq m}$ is
implied by the following result.

\begin{proposition}[Non-degeneracy for critical sets]\label{P:Surjective crits}
Let $u_1,\ldots ,u_m\in \R^n$ be $m$ distinct points. Then, the functions
\[\{\w_k\, e^{i\inprod{u_\ell}{\w}}:\;1\leq \ell \leq m,~~1\leq k\leq n\}\]
are linearly independent on $S^{n-1}$. 
\end{proposition}
\begin{proof}
Suppose 
\begin{equation}\label{E:SurjGoal}
 \sum_{\ell=1}^m \sum_{k=1}^n a_{\ell,k}\;\w_k e^{i\inprod{u_\ell}{\w}} \equiv 0,\qquad
 \w \in S^{n-1}
\end{equation}
with $\vec{a_\ell}:=\lr{a_{\ell,k}}_{k=1}^n $ not all zero. We begin by considering $n\geq 3.$ In this case, we use that the degree $k$ zonal harmonic
$Z_k(\w, \cdot )$ is highly peaked at $\w$ (see
\eqref{E:Zonal-Props}). After mutiplying \eqref{E:SurjGoal} by
$e^{i\inprod{u}{\w}}$ for an appropriately chosen $u,$ we may assume 
\begin{equation}\label{E:non-degen}
\vec{a}_\ell\neq 0 \quad \tf \quad  \inprod{\vec{a}_\ell}{u_\ell} \neq 0,\qquad\forall \ell
\end{equation}
and that the points $\widehat{u}_\ell= u_\ell/|u_\ell|\in S^{n-1}$ are not
antipodal and are distinct. Let $Q_d$ be a degree $d\neq 0$ harmonic homogeneous polynomial on
$\R^n.$ We have 
\begin{equation}\label{E:int0}
\sum_{k=1}^n \sum_{\ell=1}^m a_{\ell,k} \int_{S^{n-1}}
e^{i\inprod{u_\ell}{\w}}\w_k Q_d(\w) d\w = 0.
\end{equation}
If $P$ is a homogeneous polynomial on $\R^n$ with degree $D$, then
(see \cite[Thm. 3]{BDS})
\begin{equation}\label{E:int}
\int_{S^{n-1}}e^{i\inprod{u}{\eta}}P(\eta)d\eta = \sum_{j=0}^{[D
  /2]} \kappa_{j,D} \,\, j_{\alpha + D -j }(\abs{u})\lr{\Delta^j P}(u),
  \end{equation}
where
$\kappa_{j,D} = \lr{\frac{i}{2}}^{D} \frac{(-1)^j
  \Gamma(\alpha +1)}{j!\Gamma(\alpha + D+ 1 - j)}$ and $j_\nu$ are
normalized Bessel functions (see \eqref{E:PWE}). Note that 
\[\Delta^j (x_k Q_d(x)) =
\begin{cases}
  x_k Q_d(x) & j = 0,\\
  \partial_{x_k} Q_d (x) & j =1,\\
  0 & \text{otherwise}.
\end{cases}
\]
Hence, plugging \eqref{E:int} into \eqref{E:int0}, with $P=\w_k Q_d$ and $D=d+1$, we find that 
\begin{equation}\label{E:sum1}
 \sum_{k=1}^n\sum_{\ell=1}^ma_{\ell,k} \lr{\frac{j_{\alpha+
      d+1}(\abs{u_\ell})}{\Gamma(\alpha + d + 2)} u_{\ell,k}Q_d(u_\ell) - \frac{j_{\alpha + d }(\abs{u_\ell})}{\Gamma(\alpha + d+1)} \dell_{x_k}Q_d(u_\ell)} = 0.
  \end{equation}
Note that, since $Q_d$ is homogeneous of degree $d$,
\begin{align*}
{\partial_{x_k}} Q_d (x)
&=x_k \; d\; |x|^{d-2}  Q_d ( \widehat{x})+ |x|^d \nabla_{\pi_{\widehat{x}}(e_k)}^{S^{n-1}} Q_d ( \widehat{x}),
\end{align*}
where $e_k$ is the $k-$th unit vector, $\widehat x=x/|x|$ and $\pi_{\w}(\vec{v})$ denotes the projection of the vector $\vec{v}$ onto the tangent fiber
$T_{\w}S^{n-1}$, for $\w\in S^{n-1}$.  Hence, 
\begin{align*}
\sum_{k=1}^n a_{\ell,k} {\partial_{x_k}} Q_d (u_\ell) 
%&= d |u_\ell|^{d-2} Q_d(\widehat{u}_\ell) \inprod{ \vec{ a_\ell}}{ u_\ell} + |u_\ell|^d \inprod{\vec{a_\ell}}{\nabla^{\R^n} Q_d(\widehat{u}_\ell)}_{\R^n}\\
&= d |u_\ell|^{d-2} Q_d(\widehat{u}_\ell) \inprod{ \vec{ a_\ell}}{ u_\ell} + |u_\ell|^d \nabla_{\pi_{\widehat{u}_\ell}(\vec{a_\ell})}^{S^{n-1}} Q_d(\widehat{u}_\ell).
\end{align*}
Substituting this into \eqref{E:sum1} and setting 
\[f_d(x):=|x|^2  j_{\alpha + d+1}(\abs{x})\frac{\Gamma(\alpha + d+1)}{d\; \Gamma(\alpha + d+2 )} -  j_{\alpha + d}(\abs{x}),\]
we obtain 
\begin{equation}\label{E:Bessel11}
\sum_{\ell=1}^m \Big[f_d(u_\ell)
\inprod{\vec{a_\ell}}{u_\ell} \abs{u_\ell}^{d-2} Q_d(\widehat{u}_\ell)
- \frac{j_{\alpha + d }(\abs{u_\ell})}{d}\abs{u_\ell}^d \nabla_{\pi_{\widehat{u}_\ell}(\vec{a_\ell})}^{S^{n-1}}Q_d(\widehat{u}_\ell)\Big] = 0.
\end{equation}
Note that for $\abs{x}\neq 0,$ \eqref{E:BesselAsymptotics} implies
that $f_d(x)\gives -1$ as $d\gives \infty.$ Take $Q_d(\cdot)=Q_{d,\ell}(\cdot)=Z_{d}(\widehat{u}_\ell, \cdot)$ to be the degree $d$ zonal harmonic centered at $\widehat u_\ell.$ Note that since $n\geq 3$ and 
$\widehat{u}_\ell$'s are not antipodal for different $\ell$,
\begin{equation}\label{E:Zonal-Props}
\lim_{d\gives \infty} Q_{d,\ell }(\widehat{u}_{\ell'}) = \delta_{\ell,\ell'}\qquad
\text{and}\qquad \lim_{d\gives \infty} 
\nabla_{\pi_{\widehat{u}_\ell}(\vec{\alpha}_\ell)}^{S^{n-1}}Q_{d,\ell}(\widehat{u}_{\ell'})= 0\qquad \forall
\ell, \ell'.
\end{equation}
Let $\ell^*=\text{argmax}\set{\abs{u_\ell}~: 1\leq \ell \leq
  m}$. Dividing \eqref{E:Bessel11} by $|u_{\ell^*}|^{d-2}$ and taking
$d \to \infty$, we find
$\inprod{\vec{a_{\ell^*}}}{u_{\ell^*}}=0$. Using
\eqref{E:non-degen}, this shows $\vec{a_{\ell^*}}=0.$ Repeating this
argument for the $m-1$ remaining points completes the proof when $n\geq 3.$

If $n=2$ we cannot use the concentration of zonal harmonics. Instead, we
argue by an explicit Fourier series computation. Suppose again that \eqref{E:SurjGoal}
holds. Write $\w_1 = \cos(\theta),\,\,\w_2 = \sin(\theta).$ After multiplying \eqref{E:SurjGoal}
by $e^{i\inprod{u}{\w}},$ we may assume that the values  $|u_\ell|$ are distinct and that $\abs{u_\ell}>2$ for all
$\ell$. We can also assume that, with $\widehat{u}_\ell = e^{i\theta_\ell},$ the arguments
$\theta_\ell$ are not rational multiplies of $\pi.$ Using the plane wave
expansion \eqref{E:PWE}, we have for $\theta \in [0,2\pi]$ that
\begin{align}\label{E:fourier}
F(\theta):=\sum_{\ell=1}^m\sum_{k=0}^\infty \lr{\frac{i\,|u_\ell|}{2}}^k \frac{j_k(|u_\ell|)}{k!}\cos\lr{k\lr{\theta
  - \theta_\ell}} \lr{a_{\ell,1} \cos(\theta) + a_{\ell,2}
    \sin(\theta)}\equiv 0.
\end{align}
We have used that $c_k=1/k!$ when $n=2.$ By assumption, there exists a unique $\ell^* :=
\text{argmax}\set{\abs{u_\ell}~:~ 1\leq \ell \leq m}.$ Extracting the $N$-th
Fourier coefficients in \eqref{E:fourier}, multiplying them by $(N-1)! \cdot \lr{\frac{i\abs{u_{\ell^*}}}{2}}^{-N+1} $, and sending $N\gives\infty$, we find that 
\[\lim_{N\gives \infty}
\inprod{\vec{a}_{\ell^*}}{e^{i(N-1)\theta_{\ell^*}}}=\lim_{N\gives \infty}
\inprod{\vec{a}_{\ell^*}^{\perp}}{e^{i(N-1)\theta_{\ell^*}}}=0,\]
where for any $v=(v_1,v_2)$ we set $v^{\perp}=(v_2, -v_1).$
Hence, we must have that $\vec{a}_{\ell^*}=0.$ Repeating this argument for the
$m-1$ remaining points completes the proof. 
\end{proof}

%==========================================================

\subsection{Finiteness of the second moment for critical
  points}\label{S:critsvariance} 
In this section we prove that if $R<\infty$ then  
\begin{equation}\label{E: bdd variance}
\E{(\#\{d\phi_\infty(u) = 0 ,\,\, \abs{u}\leq R \})^2} <\infty.
\end{equation}
{Here, and in what follows, {for a map $\psi:\R^n \to \R^k$} we write  
\[
d\psi:=(\dell_{u_1}\psi, \dots, \dell_{u_n}\psi)\qquad \text{and}\qquad  \nabla_{\w}d\psi=(\nabla_{\w}\dell_{u_1}\psi, \dots, \nabla_{\w}\dell_{u_n}\psi),
\]
where $\nabla_{\w}$ is the directional derivative in the direction of $\w\in \R^n$.}
The results in Section \ref{S:KR crits} show that we may apply the Kac-Rice
formula to the moments of the counting measure of
$(d\phi_\infty)^{-1}(0).$ Hence, \eqref{E: bdd variance} is equivalent to showing
\begin{equation}\label{E:integrable}
 Y_{_{\!2,d\phi_\infty}}(u,v)\Den_{(d\phi_\infty(u),d\phi_\infty(v))}(0,0) \in L^1_{\text {loc}}(\R^n \times \R^n),
 \end{equation}
where $ Y_{_{\!2,d\phi_\infty}}$ is as in \eqref{E:KR}. Note that the density
$\Den_{(d\phi_\infty(u), d\phi_\infty(v))}(0, 0)$ blows up at the
diagonal $u=v$, so \eqref{E:integrable} is not immediate. 
{
Instead, it follows from Proposition \ref{P:L1Density}, Proposition \ref{P: bound of numerator}, Lemma \ref{L:infty-cond}, and the fact that $|u-v|^{2-n}$ is in $L^1_{\text{loc}}$. We state Propositions \ref{P:L1Density}  and \ref{P: bound of numerator} below so that they can be applied to $\psi\in \set{\phi_\leb^x,d\phi_\leb^x,\phi_\infty,d\phi_\infty}.$ In this section we  only need the case $\psi=d\phi_\infty,$ but in the proof of Theorems \ref{T:Global}-\ref{T:LocalCrits} we will need the others as well.

%\begin{proposition}\label{P:L1Density} Suppose $\psi:B_R\gives \R$ is a smooth centered Gaussian field, where $B_R$ is a ball of radius $R\in (0,\infty]$ in $\R^n$. Suppose that for every $u\in B_R$, the joint distribution of the vector $\lr{d\psi(u), \nabla_{\w}d}\psi(u)$  is uniformly non-degenerate:
%\begin{equation}\label{E:12-derivs}
%    \inf_{u\in B_R}\inf_{\w\in S^{n-1}}\det\Cov \lr{d\psi(u),\nabla_\w d\psi(u)}=c>0,
%\end{equation}
%where $\nabla_\w$ is the directional derivative in the direction of $\w.$ % Assume also that the variance of the third derivatives of $\psi$ is uniformly bounded:
%\[\sup_{u\in B_R} \Cov\lr{d^3\psi(u)}=C<\infty.\]
 %Then, as $|u-v| \to 0$ we have
%\[\Den_{(d\psi(u),d\psi(v))}(0,0) =O\big(\abs{u-v}^{-n}\big),\]
%where the implicit constant depends only on $c$ and $n.$
%\end{proposition}

\begin{proposition}\label{P:L1Density} Fix $k =1,\ldots, n.$ Suppose $\psi:B_R\gives \R^k$ is a smooth centered Gaussian field, where $B_R$ is a ball of radius $R\in (0,\infty]$ in $\R^n$. Suppose that for every $u\in B_R$, the joint distribution of the vector $\lr{\psi(u), \nabla_{\w}\psi(u)}$ is uniformly non-degenerate:
\begin{equation}\label{E:12-derivs}
    \inf_{u\in B_R}\inf_{\w\in S^{n-1}}\det\Cov \lr{\psi(u),\nabla_\w \psi(u)}=c>0,
\end{equation}
where $\nabla_\w$ is the directional derivative in the direction of $\w.$ % Assume also that the variance of the third derivatives of $\psi$ is uniformly bounded:
%\[\sup_{u\in B_R} \Cov\lr{d^3\psi(u)}=C<\infty.\]
 Then, as $|u-v| \to 0$,
\[\Den_{(\psi(u),\psi(v))}(0,0) =O\big(\abs{u-v}^{-k}\big),\]
where the implicit constant depends only on $c$, {$n$, and $k.$}
\end{proposition}

\begin{proposition}\label{P: bound of numerator}
Fix $k=1,\ldots, n.$ Suppose $\psi:B_R\gives \R^k$, { $\psi=(\psi_\alpha)_{\alpha=1}^k$,} is a smooth centered Gaussian field, where $B_R$ is a ball of radius $R\in (0,\infty]$ in $\R^n$ satisfying the following assumptions:
\begin{align*}
    (A1)&\quad\,\forall u,v\in B_R,\, u\neq v,\, \text{ the covariance matrix of }\lr{\psi(v),\psi(u)} \text{is invertible:}\\
    &\qquad \qquad \qquad \qquad \qquad\lr{\begin{array}{cc}
        F(v,v) & F(v,u) \\
        F(u,v) & F(u,u)
    \end{array}}^{-1}~~~\text{exists},\\
    &\quad \text{where }\lr{F(u,v)}_{\alpha,\beta}:=\Cov\lr{\psi_\alpha(u),\psi_\beta(v)},\quad \alpha, \beta=1, \dots, k.\\
    (A2)&\quad\text{The smallest eigenvalue for the covariance matrix of }\psi(v)\text{ given }\psi(u)=0\\
    &\quad\text{is bounded below by }\abs{u-v}^2.\text{ That is, }\exists c>0 \text{ such that }\forall \xi\in \R^k,\text{ with } \norm{\xi}=1,\\
     &\qquad\qquad\qquad \xi^T\lr{F(v,v)-F(v,u)F(u,u)^{-1}F(u,v)}\xi ~\geq~ c\abs{u-v}^2.\\
    (A3)%&\quad \forall 1\leq j,k\leq n,\, 
    &\quad \text{The entries of }F(u,v) \text{ are uniformly bounded in }C^m(B_R\x B_R)\text{ for all }m.
    %\partial_{u_j}F(u,v),\partial_{u_j}\partial_{u_k}F(u,v)\text{ and the}\\
    %&\quad \text{covariance matrix of }( \partial_{v_i}\partial_{v_j}\psi(v))_{1\leq i \leq j\leq n}\text{ are uniformly bounded for } u,v\in B_{R}.
 %   (A4)&\quad \forall u,v\in B_{R},\, \text{ the entries of} F(u,u),\partial_{u_j}F(u,v)|_{u=v}\text{ have uniformly bounded}\\
%    &\quad \text{derivatives}:\\
%    & \sup_{\substack{0<\abs{u-v}\leq 1\\ u\in B_{R-1}}}\abs{u-v}^{-1}\max\left\{\bigg|F(u,v)|_{u=v} - F(u,v)\bigg|,
%    \quad \bigg|\partial_{u_j}F(u,v)|_{u=v} - \partial_{u_j}F(u,v)\bigg|\right\}\leq C
\end{align*}
%where non-degenerate means has an invertible covariance matrix. 
Define $Y_{2,\psi}(u,v)$ as in \eqref{E:KR}. Then, there exists $C>0,$ depending only on $c$ and the $C^3$ norm of the entries of $F(u,v)$ on $B_R\x B_R$, so that, uniformly over $u\in B_R$, as $u\gives v$,
\begin{equation}\label{E: conditional exp 2}
  Y_{_{\!2,\psi}}(u,v)\leq \begin{cases}
    C\, \abs{u-v}^2 &\quad k=n,\\
    C\,,&\quad k < n.
  \end{cases}
\end{equation}
\end{proposition}

}

{

\begin{Lem}\label{L:infty-cond}
The fields $\phi_\infty$, $d\phi_\infty$ satisfy the hypotheses of Propositions \ref{P:L1Density} and \ref{P: bound of numerator} with $R=\infty.$
\end{Lem}
\begin{proof}
The proofs for $\phi_\infty,d\phi_\infty$ are  essentially the same, {so we only} provide details for the latter. The field $\phi_\infty$ is isotropic and translation invariant. Thus, $d\phi_\infty(u)$ is independent of $\nabla_\w d\phi_\infty(u)$ for every $u,\w$, and the distribution of both is independent of $u,\w.$ In particular, {writing $u=(u_1, \dots, u_n)$,}
\[\inf_{u\in B_R}\inf_{\w \in S^{n-1}}\det\Cov \lr{d\phi_\infty(u),\nabla_{\w}d\phi_\infty(u)}=\det \lr{\Cov\lr{d\phi_\infty(0)}}\det\lr{\Cov\lr{\partial_{u_1}d\phi_\infty(0)}}.\]
A direct computation using that $\int_{S^{n-1}}\w_i\w_j d\w=\frac{1}{n} \delta_{i,j}$ shows that
\[\Cov\lr{d\phi_\infty(0)}=\frac{1}{n}\mathrm{Id},\qquad 
\text{and} 
\qquad \Cov\lr{\partial_{u_1}d\phi_\infty(0)}=L,\]
where {$L$ is the diagonal matrix}
\begin{equation}\label{E:M-def}
     L=\mathrm{Diag}\lr{\kappa_4,\kappa_{2,2},\ldots, \kappa_{2,2}},\quad \kappa_4:=\int_{S^{n-1}}\w_1^4d\w,\quad \kappa_{2,2}:=\int_{S^{n-1}} \w_1^2\w_2^2d\w.
 \end{equation}
 Thus, $\phi_\infty$ satisfies the hypotheses of Proposition \ref{P:L1Density}. To check the hypotheses of Proposition \ref{P: bound of numerator},  note that the invertibility of the covariance matrix of $\lr{d\phi_\infty(v),d\phi_\infty(u)}$ is the content of Proposition \ref{P:Surjective crits}, verifying (A1) (see also the paragraph before the statement of Proposition \ref{P:Surjective crits}). Next, to check (A3), recall \eqref{E:LimitDef} and note that the $(\alpha,\beta)$ entry of $F(u,v)$ is
\[(2\pi)^{n/2}\partial_{u_\alpha} \partial_{v_\beta} \lr{J_{\frac{n-2}{2}}\lr{|u-v|}~/~|u-v|^{\frac{n-2}{2}}}.\]
This is a real analytic function that decays like $|u-v|^{-1/2}$ at infinity and hence has uniformly bounded derivatives of all orders for $u,v\in \R^n.$ Finally, we check (A2) by a direct computation. Since $\phi_\infty$ is isotropic and stationary, we may assume $u=\lr{u_1,0,\ldots,0}$ and $v=\lr{v_1,0,\ldots, 0}.$ For each $\alpha,\beta=1,\ldots, n$, {Taylor expanding the exponential below we have}
\begin{align*}
\lr{F(u,v)}_{\alpha,\beta}&=\partial_{u_\alpha}\partial_{u_\beta}\int_{S^{n-1}} e^{i\lr{u-v}\w}d\w =\int_{S^{n-1}}e^{i(u-v)\w} \w_\alpha\w_\beta d\w\\
&=\int_{S^{n-1}} \w_\alpha\w_\beta d\w - \int_{S^{n-1}}(u_1-v_1)^2\w_1^2 \w_\alpha\w_\beta d\w + O(\abs{u-v}^4).\end{align*}
Using that $\int_{S^{n-1}}\w_\alpha\w_\beta d\w=\frac{1}{n} \delta_{\alpha,\beta}$ yields
\begin{equation}\label{E:cov-inv}
   F(v,v)-F(u,v)F(u,v)^T=\abs{u-v}^{2}2 L\lr{\mathrm{Id}+O\lr{\abs{u-v}^2}},
 \end{equation}
where $L$ is as above. This completes the proof that (A2) holds and hence the proof of Lemma \ref{L:infty-cond} as well.
\end{proof}
}

\subsection{Proof of Proposition \ref{P:L1Density}} Writing $\Cov_{_{\!\psi}}(u,v):=\Cov(\psi(u), \psi(v)),$ we have
\[\Den_{(\psi(u),\psi(v))}(0,0)=\lr{2\pi}^{-n}\det (\Cov_{_{\!\psi}}(u,v))^{-1/2}.\] 
Proposition \ref{P:L1Density} is hence equivalent to proving that there exists $C$ such that for all $u \in \R^n$
\begin{equation}\label{E:L1Goal}
\frac{\sqrt{\det \Cov_{_{\!\psi}}(u,v)}}{\abs{u-v}^k}\geq C\qquad \text{as
}v\gives u.
\end{equation}
Write
\[\psi(u) = \sum_{j=1}^\infty a_j \psi_j(u) = \inprod{\vec{a}}{\Psi(u)}_{\ell_2},\qquad  \Psi(u)=\lr{\Psi_1(u),\ldots, \Psi_k(u)}\]
where $\Psi_i(u)=\lr{\psi_{i,j}(u)}_{j=1}^\infty$ is a vector of smooth  functions, whose existence is guaranteed by the Karhounen-Loeve Theorem. Note that 
\[\Cov_{_{\!\psi}}(u,v)=\text{Gram}\lr{\Psi_1(u),\ldots, \Psi_k(u),\, \Psi_1(v),\ldots, \Psi_k(v)},\]
where for any vectors $w_i$ in some inner product space $\text{Gram}(w_1,\ldots, w_\ell)=\lr{\inprod{w_i}{w_j}}_{1\leq
  i,j\leq \ell}$ is the Gram matrix. 
By the Gram Identity
\begin{align*}
\sqrt{\det  \Cov_{_{\!\psi}}(u, v)}=\norm{\Psi(u)\wedge \Psi(v)},\qquad \Psi(u):=\Psi_1(u)\wedge\cdots
\wedge \Psi_k(u)
\end{align*}
We have
\begin{align}
\frac{\sqrt{\det \Cov_{_{\!\psi}}(u, v)}}{\abs{u-v}^{k}}=\frac{\norm{\Psi(u)\wedge \Psi(v)}}{\abs{u-v}^{k}} 
=\norm{\Psi(u)\wedge \frac{\Psi(u)-\Psi(v)}{\abs{u-v}^k}}, \label{E: wedge}
%&=\abs{u-v}^{{k}}\norm{G_0(u)\wedge \nabla_w B_0(v)} +o(1),
\end{align}
where we have abbreviated
$\Psi(u)-\Psi(v)=\wedge_{i=1}^{k}
(\Psi_i(u)-\Psi_i(v))$. Fix $u$ and suppose
$v\gives u$ through a sequence that achieves the liminf in \eqref{E:
  wedge}. By passing to a further subsequence $\{v_j\}_j$ with $v_j \to u$, we may assume that
there exists $w\in S^{n-1}$ so that
\[w = \lim_{j \to \infty}\frac{u-v_j}{|u-v_j|}.\]
Since $\psi$ is smooth, we therefore have  
\[\liminf_{v\gives u}\frac{\sqrt{\det(\Cov_{_{\!\psi}}(u, v))}}{\abs{u-v}^{k}}= \norm{\Psi(u)\wedge \nabla_w
  \Psi(u)}=\sqrt{\det \Cov(\psi(u), \nabla_w\psi(u))},\]
  where we have once again used the Gram Identity and have set
$\nabla_w \Psi(u):= \bigwedge_{j=1}^k \nabla_w\Psi_j(u).$
The assumption \eqref{E:12-derivs} shows that the right hand since is uniformly bounded below over $u,\w,$ completing the proof. \hfill $\square$

\subsection{Proof of Proposition \ref{P: bound of numerator}}
{ %TO DO:
%\begin{itemize}
%    \item State and explain that the result holds for both $\phi_\infty$ and $\phi_\leb$.
%    \item Explain why we can ask $F(u,u)=I$.
%\end{itemize}

\bigskip
\noindent According to \eqref{E:KR}, $Y_{_{\!2,\psi}}(u,v)$ has the form
\begin{align}\label{E:Y2}
 \E{\lr{\det\lr{(d\psi(u))^*d\psi(u)}\det\lr{(d\psi(v))^*d\psi(v)}}^{1/2}\;\big|\; \psi(u)=\psi(v)=0}.
\end{align}
%\begin{align}\label{E:Y2}
 %Y_{_{\!2,d\psi}}(u,v)=\E{\abs{\det \text{Hess}\, \psi (u)}\abs{\det \text{Hess}\,\psi(v)} \;\big|\; d\psi(u)=d\psi(v)=0}.
%\end{align}
 Using that $2ab\leq a^2 + b^2,$ it is sufficient to show that there exists $C>0$ so that 
\begin{equation}\label{E:KR-crits-goal-1}
    \sup_{0<\abs{u-v}\leq 1}\E{\det\lr{(d\psi(v))^*d\psi(v)}\;\big|\; \psi(u)=\psi(v)=0 }\leq \begin{cases}
    C\, \abs{u-v}^2&\quad k=n,\\
    C\,&\quad k < n.
  \end{cases}
\end{equation}
\noindent Our proof of \eqref{E:KR-crits-goal-1} relies on the following result.
\begin{Lem}\label{L:cov-deriv-bounds}
  Suppose $\psi$ satisifies the hypotheses of Proposition \ref{P: bound of numerator}. Then, 
  \begin{equation}\label{E:3-derivs}
\sup_{\w\in S^{n-1}}\sup_{0<\abs{u-v}\leq 1}\sup_{t\in (0,1)}\Cov\lr{\nabla_\w^2 \psi(tu + (1-t)v)~|~\psi(u)=\psi(v)=0}~<~\infty,
\end{equation}
where the bound applies to each entry of the covariance matrix and $\nabla_\w$ is the directional derivative in the direction $\w$. Similarly, 
\begin{equation}\label{E:2-derivs}
\sup_{0<\abs{u-v}\leq 1}\Cov\lr{d\psi(v)~|~\psi(u)=\psi(v)=0}~<~\infty.
\end{equation}
\end{Lem}}
\bigskip
\noindent We prove Lemma \ref{L:cov-deriv-bounds} in Section \ref{S:cov-deriv}  by a direct computation. Before doing so, we show why it implies \eqref{E:KR-crits-goal-1}. \\
The case $k<n,$ follows immediately by combining \eqref{E:2-derivs} and Lemma \ref{L:Gaussian-integrals} below. To treat the slightly more technical case $k=n,$ note that
\begin{equation}\label{E: transpose} 
\det\lr{(d\psi(v))^*d\psi(v)}=\norm{ \partial_{v_1} \psi (v)\wedge\dots\wedge  \partial_{v_n} \psi (v) }.
\end{equation}
 Write $w=\frac{u-v}{\abs{u-v}}.$ Note that since $\nabla_\w \psi_j(v)\cdot \w$ is the projection of $\nabla \psi_j(v)$ onto the direction of $\w,$ the $n$ vectors
\[\vec{B}_j:=\nabla \psi_j(v)-
\nabla_\w \psi_j(v)\cdot \w ,\qquad j=1,\ldots,n\]
are contained in the orthogonal complement to $w$ in $\R^n.$ Hence, there exist $a_j\in \R$, not all zero, so that
\[\sum_{j=1}^n a_j \vec{B}_j=0\qquad \text{and}\qquad \max_{j=1,\ldots, n}\set{\abs{a_j}}=1.\]
Without loss of generality we will assume $a_1\neq 0$. Since wedge products are unchanged after adding to a column a linear combination of all columns, we therefore have
\begin{align}\label{E:dethess}
  \det\lr{(d\psi(v))^*d\psi(v)}&= \Big\|{\sum_{j=1}^n a_j
                                       \nabla_w \psi_j(v)\cdot
                                       w\wedge  d \psi_2 (v)\wedge\cdots\wedge  d\psi_n (v) }\Big\|^2.
\end{align}
 {
Since $\norm{\eta \wedge \zeta}\leq \norm{\eta}\norm{\zeta},$ we conclude using again that $ab\leq a^2+b^2$ and Cauchy-Schwarz, that the conditional expectation in \eqref{E:KR-crits-goal-1}  is bounded above by the product
\begin{align}\label{E:dethess-2}
  &2\sum_{j=1}^n \mathbb  E\bigg[  \abs{\nabla_w \psi_j(v)}^4~\big|~\psi(u)=\psi(v)=0\bigg]^{1/2}\\
\label{E:dethess-3}&\qquad\qquad \qquad\x\quad\E{  \Big\|{ d \psi_2 (v)\wedge\cdots\wedge d \psi_n (v) }\Big\|^4~\big|~\psi(u)=\psi(v)=0}^{1/2}.
\end{align}
We will see that the first term is $O(|u-v|^2)$ and the second term is $O(1).$ To do that, we need the following fact. 
\begin{Lem}\label{L:Gaussian-integrals}
Fix $N\geq 1$ and suppose $A$ is a symmetric positive semi-definite (PSD) $N\x N$ matrix. Denote by $\mu_{_{\!A}}$ the associated centered Gaussian measure on $\R^N.$ Then, for any measurable function $f:\R^N\gives \R$ we have
\begin{equation}\label{E:delta-var}
    \int_{\R^N} f(x)d\mu_{_{\!A}}(x)~=~ \int_{\R^N} f(A^{1/2}x) d\mu_{_{\mathrm{Id}}}(x).
\end{equation}
In particular, denote by $\norm{A}_\infty$ the magnitude of the largest entry in $A.$ For each $d$, there exists a polynomial $p(t)$ of degree $d$ whose coefficients depend only on $d$ and $N$ so that 
\begin{equation}\label{E:gaussian-int-est}\abs{f(x)}\leq C \lr{1+\norm{x}}^d\;\; \forall x \in \R^N\quad \Longrightarrow
\quad
    \int_{\R^N} f(x)d\mu_{_{\!A}}(x)~\leq~ C\cdot p(\norm{A}_\infty).
  \end{equation}
 % In particular, for a fixed $f$, the left hand side of \eqref{E:gaussian-int-est} is uniformly bounded over all PSD matrices $A$ whose entries are uniformly bounded by a given constant. 
\end{Lem}
\begin{proof}
 Write $A^{1/2}$ for the PSD square root of $A$. If $x$ is distributed according to the standard Gaussian measure $\mu_{\mathrm{Id}}$ on $\R^N$, then $A^{1/2}x$ is distributed according to $\mu_{_{\!A}}.$ This proves \eqref{E:delta-var}. Relation \eqref{E:gaussian-int-est} then follows from the fact that $\mu_{_{\mathrm{Id}}}$ has bounded moments of all orders. 
\end{proof}
We emphasize that Lemma \ref{L:Gaussian-integrals} holds even when $A$ is not strictly positive definite. Note that relation \eqref{E:2-derivs}, combined with \eqref{E:gaussian-int-est} shows immediately that \eqref{E:dethess-3} is uniformly bounded over the set $0<\abs{u-v}\leq 1$. To complete the proof of {Proposition \ref{P: bound of numerator}}, we check that \eqref{E:3-derivs} implies that the expression in \eqref{E:dethess-2} is $O(\abs{u-v}^2)$. To see this, note that if $ d\psi(u)=d\psi(v)=0 $, then the Mean Value Theorem yields, for each {$j=1, \dots, n$}, the existence a point $c_j=c_j(\psi,u,v)$ in the interior of the line segment between $u$ and $v$ so that 
\[\nabla_{w}\psi_j(c_j)=0.\]
Hence, for each $j=1,\ldots, n,$ we may write
\begin{align}
 \abs{\nabla_w \psi_j(v)}^4&= \abs{\int_0^1 |v-c_j|
                                              \nabla_{\w}^2 \psi_j((1-t)c_j
                                              + tv)dt}^4 \notag\\
&\leq \abs{u-v}^4\int_0^1\abs{ \nabla_{w}^2\psi_j((1-t)u+t
  v)}^4dt. \label{E:u-v bound}
\end{align}
Thus, using Jensen's inequality, we find that  for all $u,v$ with $0<|u-v|\leq1$, $\E{ \abs{\nabla_\w^2 \psi_j(v)}^4~|~\psi(u)=\psi(v)=0}$ is bounded above by
\begin{equation}\label{E:u-v bound2}
\abs{u-v}^4\sup_{t\in (0,1)}\E{\abs{\nabla_\w^2 \psi_j((1-t)u +t v)}^4~|~\psi(u)=\psi(v)=0}.
\end{equation}
Combining \eqref{E:3-derivs} with \eqref{E:u-v bound}, \eqref{E:u-v bound2}, and \eqref{E:gaussian-int-est} therefore confirms \eqref{E:dethess-2} is bounded by a uniform constant times $\abs{u-v}^2,$ as claimed. \hfill $\square$

\subsection{Proof of Lemma \ref{L:cov-deriv-bounds}}\label{S:cov-deriv}

%\begin{Lem}\label{L:norm}
%There exists a universal $C>0$ such that 
%\begin{equation}\label{E:KR-crits-goal-2}
%\E{\norm{w \wedge  \nabla\partial_{u_2}\phi(v)\wedge\cdots \wedge  \nabla\partial_{u_n}\phi(v)}^2\;\big|\; d\phi(u)=d\phi(v)=0 }\leq C. 
%\end{equation}
%\end{Lem}

%\begin{Lem}\label{L:int}
%There exists a universal $C>0$ such that for every $j=1, \dots, n$, and all $t\in [0,1]$
%\begin{equation}\label{E:KR-crits-goal-2}
%\E{\abs{\nabla_\omega^2 \partial_{u_j} \phi(p_{u,v}(t))}^2\;\big|\; d\phi(u)=d\phi(v)=0 }\leq C, 
%\end{equation}%
%where $p_{u,v}(t)=tv+(1-t)u$.
%\end{Lem}

%%%%%%%%%%%%%%%%%%%%%%%%%%%%%%%%%%%%%%%%%%%%%%%%%%%%%%%%%%%%
%\ \\
%\noindent \emph{Proof of Lemma \ref{L:norm}.}
The proofs of \eqref{E:2-derivs} and \eqref{E:3-derivs} are identical. We therefore give the details for \eqref{E:2-derivs} and omit the proof of \eqref{E:3-derivs}. Fix $u\neq v$, and recall that $F(u,v)=d_ud_v\E{\psi(u)\psi(v)}=\Cov\lr{d\psi(u),\,d\psi(v)}$. %We will define the matrix $\dell_{u_k} F(u,v)$ by
%\[(\dell_{u_k}F(u,v))_{i,j}=(\dell_{u_k}\partial_{u_i}\partial_{v_j}\E{\psi(u)\psi(v)})_{i, j}.\] 
Note that a linear change of coordinates $(u,v)\mapsto (Au, Av)$ transforms $F(u,v)$ as follows:
\[F(Au, Av)=A^TF(u,v)A.\]
By Assumption (A1), the distribution of $d\psi(u)$ is non-degenerate. Hence, taking $A$ to be the square root of the inverse of the covariance matrix of this vector, we may work in coordinates so that 
\begin{equation}\label{E:coord-choice}
    F(u,u)=\mathrm{Id}.
\end{equation}
This is not really essential but will simplify the notation below. Assumption (A1) also guarantees that $\lr{\psi(u),\psi(v)}$ is non-degenerate and hence its covariance matrix is invertible. Thus, 
\begin{align*}
 & \Cov \lr{( d\psi_1(v), \dots,  d\psi_k(v))\;\big|\; \psi(u)=\psi(v)=0 }\\
 & \qquad\qquad\qquad \qquad\qquad\qquad = \Cov((d\psi_1(v), \dots,  d\psi_k(v)))-R(u,v), 
\end{align*}
where $R(u,v)$ is the  $nk \times nk$ matrix
\[R(u,v)\!=\!
\lr{\begin{array}{ccc}
    \!\dell_{u_1} F(u,u) & \cdots & \dell_{u_n} F(u,u) \\
    \!\dell_{u_1} F(u,v) & \cdots & \dell_{u_n} F(u,v)
\end{array}}^{\!T}
\!\!\!\!\lr{\begin{array}{cc}
    X & Y \\
    W & Z 
\end{array}}
\!\!\lr{\begin{array}{ccc}
    \!\dell_{u_1} F(u,u) & \cdots & \dell_{u_n} F(u,u) \\
    \!\dell_{u_1} F(u,v) & \cdots & \dell_{u_n} F(u,v)
\end{array}}
\]
with
\begin{equation}\label{E:XYZW}
\lr{\begin{array}{cc}
    X & Y \\
    W & Z 
\end{array}}=
\lr{\begin{array}{cc}
    F(u,u) & F(u,v) \\
    F(v,u) & F(v,v) 
\end{array}}^{-1}.
\end{equation}
Assumption (A3) shows that it remains only to bound the entries of the matrix $R(u,v)$, which can be thought of as a collection of $k \times k$ blocks $(R(u,v))_{ij}$, with $i,j=1, \dots, n$, given by
\[
\dell_{u_i} F(u,u)^T \Big(X \dell_{u_j} F(u,u)+Y \dell_{u_j} F(u,v)\Big)+\dell_{u_i} F(u,v)^T \Big(W \dell_{u_j} F(u,u)+Z \dell_{u_j} F(u,v)\Big).
\]
Next, let $M_\ell={M_\ell(v)}$ and $\ep_\ell=\ep_\ell(u,v)$ be $k \times k$ matrices so that after Taylor expansion at $u=v$ we have
\[
\dell_{u_\ell} F(u,v)= M_\ell + \ep_\ell, \qquad \qquad M_\ell=\dell_{u_\ell} F(u,v)\big|_{u=v},
\]
 for all $\ell=1, \dots, n$. By Assumption (A3) the entries of $M_\ell(u)$ are uniformly bounded by a constant and  the entries of $\ep_\ell$ are uniformly bounded by a constant times $\abs{u-v}$:
 \begin{equation}\label{E:M-ep}
     M_\ell=O(1),\qquad \ep_\ell(u,v)=O(|u-v|).
 \end{equation}
 Then, for $i=1,\ldots, k,j=1, \dots, n$,
\[
(R(u,v))_{ij}=M_i^T(X+Y+W+Z) M_j+ M_i^T (Y+Z)\ep_j + \ep_i^T(W+Z)M_j+\ep_i^T Z \ep_j.
\]
Lemma \ref{L:cov-deriv-bounds} therefore reduces to showing that for every $i,j=1, \dots, n$, we have
\begin{equation}\label{E:Rij-goal}
(R(u,v))_{ij}=O(1).
\end{equation}
To do this, note that using \eqref{E:coord-choice} and the usual block matrix inversion formula, we have
\begin{align*}
    X&= \mathrm{Id} + F(u,v) Z F(v,u), &  W&=-ZF(v,u),\\ 
    Y&=-F(u,v)Z,&    Z&=\Big(F(v,v)-F(v,u)F(u,v)\Big)^{-1}.
\end{align*}
Using the expressions for $X, W$ it is straight forward to deduce that
\[
    Y+Z= (\mathrm{Id}-F(u,v))Z,\qquad 
    W+Z=Z(\mathrm{Id}-F(v,u))
\]
as well as
\[  X+Y+Z+W~=~\mathrm{Id}+ (\mathrm{Id}-F(u,v)) Z (\mathrm{Id}-F(v,u)).\]
By Taylor expansion and Assumption (A3),
\[\mathrm{Id}-F(u,v)=O(|u-v|), \qquad \mathrm{Id}-F(v,u)=O(|u-v|),\]
uniformly for $u,v\in B_R.$ Assumption (A2) shows that for every $\xi\in \R^n$, we have
\[\xi^TZ^{-1}\xi \geq c \norm{\xi}^2\abs{u-v}^2.\]
Hence, the operator norm of $Z$ is bounded above by $c^{-1}\abs{u-v}^{-2}$ and relation \eqref{E:Rij-goal} follows by combining \eqref{E:M-ep} with the explicit expression for $ (R(u,v))_{i,j}$ just above \eqref{E:Rij-goal}. This completes the proof of \eqref{E:2-derivs}.

}

%together with the 
%Therefore, the proof of the estimates in \eqref{E:sums}  will then follow after we show that 
%\begin{equation}\label{E:Z}
%    Z=O(|u-v|^{-2}).
%\end{equation}
%We prove \eqref{E:Z} below and first show that  \eqref{E:sums2} also holds. First, note that 
%\[
%  X+Y+Z+W=\mathrm{Id}+ F(u,v)ZF(v,u)-F(u,v)Z- ZF(v,u) +Z.
%  \]
%Therefore, using that $F(u,v)=\mathrm{Id} + \delta$ and $F(v,u)=\mathrm{Id} + \delta^T$ where $\delta=O(|u-v|)$, a direct computation yields
%\[
%  X+Y+Z+W=I+ \delta Z \delta^T.
%  \]
%  The estimate in \eqref{E:sums2} then follows from \eqref{E:Z}. 
%Finally, to prove \eqref{E:Z} simply note that if we Taylor expand at $v=u$,
%\[
%F(u,v)=\mathrm{Id}+ D(u,v)+O(|u-v|^2),
%\]
%where 
%\[(D(u,v))_{\alpha,\beta}=\sum_{i=1}^n \partial_{u_i}F_{\alpha, \beta}(u,u)(u_i-v_i),\]
%for $\alpha, \beta=1, \dots, n$.
%Since $F(u,v)=F(v,u)^T$, it follows that 
%\[
%F(v,v)=\mathrm{Id}+ D(u,v)+D(u,v)^T+O(|u-v|^2).
%\]
%Therefore, 
%\begin{align*}
%Z^{-1}&=F(v,v)-F(u,v)F(v,u)\\
%&=\mathrm{Id}+ D(u,v)+D(u,v)^T - \Big([I +D(u,v)][I
%+D(u,v)^T] \Big)  +O(|u-v|^2)\\
%&=O(|u-v|^2).
%\end{align*}
%Note however, that 
%\[\det(Z^{-1})=\det\lr{\Cov\lr{d\psi(u),d\psi(v)}}\geq c \abs{u-v}^2,\]
%where the inequality used Assumption (A2). 
%By Lemma \ref{L:1} below, when $\psi=\psi_\infty$ \cs FINISH\cs
\qed
%\ \\

%%%%%%%%%%%%%%%%%%%%%%%%%%%%%%%%%%%%%%%%%%%%%%%%%%%%%%

%==========================================================
%==========================================================
%==========================================================
%==========================================================
\section{Local results: Proof of Theorems \ref{T:LocalZeros} and \ref{T:LocalCrits}}\label{S: Limit theorem}
We now prove Theorems \ref{T:LocalZeros} and \ref{T:LocalCrits}. Let $\phi_\leb \in \MRW_\leb(M,g)$ and $x \in \mathcal{IS}(M,g)$. Fix also a non-negative function $r_\leb$ which satisfies $r_\leb=o(\leb)$ as $\leb\gives \infty$ and converges to a non-zero limit: $r_\leb\gives r_\infty \in (0,\infty].$

{We prove Theorem \ref{T:LocalZeros} for $r_\infty<\infty$ in Section \ref{S:small-ball-local-zeros},  Theorem \ref{T:LocalCrits} for $r_\infty<\infty$  in Section \ref{S:yet another section}, and Theorems \ref{T:LocalZeros} and \ref{T:LocalCrits} for $r_\infty<\infty$ in Section \ref{S:both thms}.  Finally, in Section \ref{S:constants} we explicitly compute the corresponding constants.}
%Note that the fields $\phi_\leb^x$ and $d\phi_\leb^x$ satisfy the hypotheses of the Kac-Rice theorem for all $\leb$ sufficiently large.

\subsection{Proof of Theorem \ref{T:LocalZeros} when $r_\infty<\infty$}\label{S:small-ball-local-zeros}
We begin with the proof of Theorem \ref{T:LocalZeros} in the case $r_\infty<\infty.$ We reproduce here the elegant argument communicated to us by G.~Peccati. We seek to show that for any bounded measurable function $\psi:B(0,r_\infty)\gives \R$ we have
\begin{equation}
\inprod{Z_{\leb, r_\leb}^x}{\psi}\quad\stackrel{d}{\longrightarrow}\quad \inprod{Z^x_{\infty, r_\infty}}{\psi}.\label{E:zero-conv-loc}
\end{equation}
To do this, we fix $R>r_\infty$ so that for all $\leb \gg 1$ we have $r_\leb<R.$ Consider the sequence of laws on $C^0(B(0,R), \R^{n+1})$, which we denote $\mu_{\lambda,R}^x,$ associated to the random fields
$J_1(\phi_\lambda^x)=\lr{\phi_\lambda^x, \nabla \phi_\lambda^x}$
restricted to the compact set $B(0,R).$ By Kolmogorov's tightness criterion, this sequence is tight in $\lambda$ since the covariance functions of the fields are uniformly smooth and we have convergence of finite dimensional distributions (c.f. e.g. \cite[Theorem 1.4.7]{Kunita}). Moreover, by the $C^\infty$ convergence of covariance functions \eqref{E:CovConv1}, Prokhorov's theorem guarantees that $\mu_{\lambda,R}^x$ converges weakly to $\mu_{\infty,R}$. Hence, Skorohod's representation theorem gives a coupling of $J_1(\phi_\lambda^x)$ and $J_1(\phi_\infty^x)$ in which $J_1(\phi_\lambda^x)$ converges pointwise almost surely to $J_1(\phi_\infty^x).$ This implies \eqref{E:zero-conv-loc}. Indeed, with probability $1$, by Bulinskaya's Lemma (cf eg Proposition 6.12 in \cite{AW}) the zero set of $\phi_\infty^x$ inside of $B(0,R)$ is non-degenerate (i.e. a smooth co-dimension $1$ submanifold). Thus, using that non-degenerate zero sets are stable under $C^1$ perturbations and dominated convergence yields the result. 

This argument does not seem to hold for the case when $r_\infty=\infty$ since the domain $B(0,R)$ is non-compact and one cannot apply Kolmogorov's tightness criterion. Moreover, such an argument cannot possibly work for critical points of $\phi_\lambda^x$ since they are not stable under $C^k$ perturbations for any $k$ (see Theorem \ref{T:LocalCrits}). 

\subsection{Proof of Theorem \ref{T:LocalCrits} when $r_\infty<\infty$}\label{S:yet another section}
We seek to apply the Kac-Rice formula for $m=1,2$ to conclude 
 \begin{equation}\label{E: moments convergence}
     \E{\C_{\leb,r_\leb}(\psi)^m} ~~\longrightarrow~~ \E{\C_{\infty, r_\infty}(\psi)^m},
 \end{equation}
 as $\leb \to \infty$. To do this, we begin by checking that the fields $d \phi_\lambda^x(u)$ satisfy the hypotheses of the Kac-Rice formula \eqref{E:KR} for all $\leb$ sufficiently large. Since the fields are smooth, we must only check the non-degeneracy condition (2).  This condition is trivial when $m=1.$ The following lemma verifies this for $m=2.$
 { \begin{Lem}\label{L:KR-crits}
 Let $r_\lambda\in (0,\infty)$ be a non-negative sequence with $\liminf r_\lambda > 0 $ and $r_\leb = o(\leb)$ both as $\leb \gives \infty$. Then, there exists $\Lambda$ so that for all $u\neq v$, and all $\leb \geq \Lambda$, the distribution of $\lr{d\phi_\leb^x(u),d\phi_\leb^x(v)}$ is non-degenerate. In particular, there exists $C,\Lambda >0$ so that for all $u\neq v$ with $u,v\in B_{r_\leb}$
 \begin{equation}\label{E:cov-bound}
 \inf_{\leb\geq \Lambda}[\det \Cov\lr{d\phi_\leb^x(u), d\phi_\leb^x(v)}]^{1/2}\geq C\abs{u-v}^n.
 \end{equation}
 \end{Lem}
 \begin{remark}
 We state Lemma \ref{L:KR-crits} so that it applies both to the case $r_\infty<\infty$ and $r_\infty=\infty.$ In this section we use only the case $r_\infty<\infty$, but in the next section we will apply it to $r_\infty=\infty.$
 \end{remark}
 \begin{proof}[Proof of Lemma \ref{L:KR-crits}]
We must check that $\phi_\leb^x$ satisfies the hypothesis \eqref{E:12-derivs} of Proposition \ref{P:L1Density}. That is, we seek to show that there exist $c>0$ and $\Lambda >0$ so that 
 \[ \inf_{\lambda \geq \Lambda}\inf_{u\in B_{r_\leb}}\inf_{\w\in S^{n-1}}\det\Cov \lr{d\phi_\infty^x(u),\nabla_\w d\phi_\infty^x(u)}=c>0.\]
The $C^\infty$ convergence of covarinace kernels \eqref{E:CovConv1} shows that 
\[\det\Cov \lr{d\phi_\leb^x(u),\nabla_\w d\phi_\leb^x(u)}=\det\Cov \lr{d\phi_\infty^x(u),\nabla_\w d\phi_\infty^x(u)}+o_\leb(1),\]
where the rate of convergence is uniform over $\w\in S^{n-1}$ and $u\in B_{r_\leb}$. 
 By Lemma \ref{L:infty-cond}, the expression $\det\Cov \lr{d\phi_\infty^x(u),\nabla_\w d\phi_\infty^x(u)}$ is bounded from below away from $0$ uniformly in $u,$ completing the proof.  
\end{proof}
 }
 
 %his for $m=2$, {\cs perhaps we should make this specific in the statement of prop 5, or write a remark\cs}.
 
 We have now verified that we can apply the Kac-Rice formula for finite $\leb$ to study the convergence in \eqref{E: moments convergence} with $m=1,2$ (that we can apply it for $\leb=\infty$ was already show in Proposition Section \ref{P:Surjective crits}). It remains only to check that the limit of the expressions for finite $\lambda$ is the expression for $\lambda=\infty.$ We will work out the details for the more complicated case $m=2.$ Note that it is enough to check the case $\psi={\bf 1}_A$, the indicator of a bounded set $A$, since linear combinations of these are dense in the space of bounded measurable functions. For such $\psi$ we have, by the Kac-Rice formula that the factorial moment
\[\frac{1}{\vol(B(0,R))^2}\mathbb E \Big[\big(\sum_{\tiny{\substack{d\phi_\leb^x(u)=0\\u\in B(0,R)}}}\psi(u)\big)\big(\sum_{\tiny{\substack{d\phi_\leb^x(u)=0\\u\in B(0,R)}}}\psi(u) - 1\big)\Big]\]
is equal to 
\begin{align}\label{E:lambda int}
\frac{1}{\vol(B(0,R))^2}\int_{(A\cap B(0,R))^2}Y_{_{\!2,d\phi_\leb^x}}\lr{u,v}\Den_{(d\phi_\leb^x(u),d\phi_\leb^x(v))}\lr{0,0}dudv,
\end{align}
where
\[Y_{_{\!2,d\phi_\leb^x}}(u, v)=\E{\left[\det \mathrm{Hess}(\phi_\leb^x)(u)\right]^{\frac{1}{2}}\left[\det \mathrm{Hess}(\phi_\leb^x)(v)\right]^{\frac{1}{2}}\,\Big|\, d\phi_\leb^x(u)=d\phi_\leb^x(v)=0}.\]

{ 
We prove that \eqref{E:lambda int} converges, as $\lambda \to \infty$, to the analogous expression for $d\phi_\infty^x$
%\begin{align*}
%&\lim_{\leb\gives \infty}\int_{(A\cap B(0,R))^2}Y_{_{\!2,d\phi_\leb}}\lr{u,v}\Den_{d\phi_\leb(u),d\phi_\leb(v)}\lr{0,0}dudv\\
%\int_{(A\cap B(0,R))^2}Y_{_{\!2,d\phi_\infty}}\lr{u,v}\Den_{d\phi_\infty(u),d\phi_\infty(v)}\lr{0,0}dudv
%\end{align*}
in two steps. First, we show that one can apply dominated covergence to the expression in \eqref{E:lambda int}. Then, we show that pointwise, when $u\neq v$, the limit of the integrand in \eqref{E:lambda int} as $\lambda \gives \infty$ is the integrand from the Kac-Rice formula applied to $d\phi_\infty.$ In short, the proof of Theorem \ref{T:LocalCrits} when $r_\infty<\infty$ reduces to the following lemma. 
\begin{Lem}\label{L:dom-conv}
Let $r_\lambda\in (0,\infty)$ be any non-negative sequence with $\liminf r_\lambda > 0 $ and $r_\leb = o(\leb)$ both as $\leb \gives \infty$. There {exist} $C>0$ and $\Lambda >0$ so that for all $\leb\geq \Lambda$ we have
\begin{equation}\label{E:Y-bound}
    \sup_{0<\abs{u-v}\leq r_\leb}Y_{2,d\phi_\leb^x}(u,v)\leq C\abs{u-v}^2,
\end{equation}
and
\begin{equation}\label{E:den-bound}
    \sup_{0<\abs{u-v}\leq r_\leb}\Den_{(d\phi_\leb^x(u),d \phi_\leb^x(v))}\leq \frac{C}{\abs{u-v}^n}.
\end{equation}
In particular, for all $\leb \geq \Lambda$  $Y_{2,d\phi_\leb}(u,v)\Den_{(d\phi_\leb^x(u),d \phi_\leb^x(v))}$ is dominated by $|u-v|^{-n+2}$, which is in $L^1$ with respect to $\vol(B_{r})^{-2}{\bf 1}_{\set{u,v\in B_r}}{dudv}$ for all $r$. Moreover, for fixed $u\neq v$,
\begin{equation}\label{E:pointwise}
\lim_{\leb\gives \infty}Y_{_{\!2,d\phi_\leb}}\lr{u,v}\Den_{\lr{d\phi_\leb(u),d\phi_\leb(v)}}\lr{0,0}=Y_{_{\!2,d\phi_\infty}}\lr{u,v}\Den_{\lr{d\phi_\infty(u),d\phi_\infty(v)}}\lr{0,0}.
\end{equation}
\end{Lem}
 \begin{remark}
 We state Lemma \ref{L:dom-conv} so that it applies both to the case $r_\infty<\infty$ and $r_\infty=\infty.$ In this section we use only the case $r_\infty<\infty$, but in the next section we will apply it to $r_\infty=\infty.$
 \end{remark}
}
\begin{proof}
{ By Lemma \ref{L:KR-crits}, there exists $\Lambda_0>0$ so that for all $\lambda \in [\Lambda_0,\infty]$, and $u\neq v$ the random variable $(d\phi_\leb^x(u),d\phi_\leb^x(v))$ is non-degenerate. Hence, in the notation of Lemma \ref{L:KR-crits},
\begin{equation}\label{E:den-cov}
    \Den_{(d\phi_\leb^x(u),d\phi_\leb^x(v))}\lr{0,0}~=~(2\pi)^{-n}\lr{\det\Cov_{d\phi_\leb^x}\lr{u,v}}^{-1/2}.
\end{equation}
Equation \eqref{E:den-bound} now follows from \eqref{E:cov-bound}. We now turn to showing \eqref{E:Y-bound} for which we need to verify that the field $d\phi_\leb^x$ satisfies  hypotheses (A1)-(A3) in Proposition \ref{P: bound of numerator} for all $\leb$ sufficiently large. We already proved in Lemma \ref{L:KR-crits} that assumption (A1) is satisfied for all $\leb$ sufficiently large. Next, assumption (A3) holds by the uniform convergence of covariance kernels \eqref{E:CovConv1} and the fact that the limiting kernel $\Pi_\infty^x(u,v)$ for $\phi_\infty^x$ has uniformly bounded derivatives of all orders. To check (A2), we recall the notation $F(u,v)=d_ud_v \E{\phi_\leb^x(u), \phi_\leb^x(v)}$ from Proposition \ref{P: bound of numerator} {with $\psi=d\phi_\lambda^x$}. Note that for all $\xi\in \R^n,$
\[\xi^T\lr{F(v,v)-F(v,u)F(u,u)^{-1}F(u,v)}\xi\geq 0\]
since $F(v,v)-F(v,u)F(u,u)^{-1}F(u,v)$ is the covariance matrix for $d\phi_\leb^x(v)$ given that $d\phi_\leb^x(u)$ vanishes. Thus, since the expression in the previous line vanishes at $u=v,$ we may Taylor expand around $u=v$ to obtain 
\[\xi^T\lr{F(v,v)-F(v,u)F(u,u)^{-1}F(u,v)}\xi = \xi^T M(u)\xi |u-v|^2 + O(\abs{u-v}^3),\]
where $M$ is the second order term in the Talyor expansion of the matrix on the left hand side (and in particular is given by a finite number of derivatives of the covariance matrix $\Pi_\leb^x(u,v)$ of $\phi_\leb^x$ evaluated at $u=v$). Hence, by the computation of $L$ in Lemma \ref{L:infty-cond}, and the $C^\infty$ convergence of covariance kernels \eqref{E:CovConv1}, we have
\[\sup_{\abs{u}<r_\leb}\abs{M(u)-2\mathrm{Diag}\lr{\kappa_4,\kappa_{2,2},\ldots, \kappa_{2,2}}}~=~o_\leb(1),
 \]
 with $\kappa_4=\int_{S^{n-1}}\w_1^4d\w$, $ \kappa_{2,2}=\int_{S^{n-1}} \w_1^2\w_2^2d\w$ as in Lemma \ref{L:infty-cond}. This proves (A2) and hence establishes \eqref{E:Y-bound}. Finally, to prove \eqref{E:pointwise}, fix $u\neq v.$ For all $\leb$ sufficiently large, we have, by the $C^\infty$ convergence of covariance kernels \eqref{E:CovConv1}, that
 \[\Cov_{d\phi_\leb}(u,v)=\Cov_{d\phi_\infty}(u,v)+o_\leb(1).\]
 Moreover, $\Cov_{d\phi_\leb}(u,v)$ is non-singular by Lemma \ref{L:KR-crits} for all $\leb$ sufficiently large. Therefore, by \eqref{E:den-cov}, we have
 \begin{align*}
     \Den_{(d\phi_\leb^x(u),d\phi_\leb^x(v))}
   =\Den_{(d\phi_\infty^x(u),d\phi_\infty^x(v))}+o_\leb(1) .
 \end{align*}
Thus, it remains to show that as $\leb \gives \infty$
\begin{equation}\label{E:Y-goal}
    Y_{_{\!2,d\phi_\leb^x}}(u, v)=Y_{_{\!2,d\phi_\infty^x}}(u, v)+o_\leb(1).
\end{equation}
To do this, write $\Sigma_{u,v,\leb}= \Cov\lr{d^2\phi_\leb(u),d^2\phi_\leb(v)~|~d\phi_\leb(u)=d\phi_\leb(v)=0}$ for the covariance matrix of the second derivatives of $\phi_\leb$ at $u,v$ given the first derivatives vanish. Note that for $u\neq v$ this covariance is well-defined since the distribution of $\lr{d\phi_\leb(u),d\phi_\leb(v)}$ is non-degenerate. Therefore, using Lemma \ref{L:Gaussian-integrals}, we may write
\[Y_{_{\!2,d\phi_\leb^x}}(u, v)=\int_{\R^N} f(\Sigma_{u,v,\leb}^{1/2}\xi)d\mu_{\mathrm{Id}}(\xi),\]
where $N=2n(n+1)$ and $f$ is a function of polynomial growth. The entries of the covariance matrix $\Sigma_{u,v,\leb},$ for our fixed $u,v$, are polynomials in $\Pi_\leb^x(u,v)$ and its derivatives. Therefore,
\[\Sigma_{u,v,\leb}=\Sigma_{u,v,\infty}+o_\leb(1).\]
The function $f$ has polynomial growth at infinity, and so we may apply dominated convergence to prove \eqref{E:Y-goal}. This completes the proof of Lemma \ref{L:dom-conv}.
}

\end{proof}

\subsection{{Proofs of Theorems \ref{T:LocalZeros} and \ref{T:LocalCrits} when  $r_\infty=\infty$}}\label{S:both thms}
We begin by proving \eqref{E:expectation}, for which it suffices to establish that $\lim_{\leb\gives \infty}\Var[Z_{\leb, r_\leb}^x(1)]=0$, which is precisely \eqref{E: equal to infty}, together with the equality
\begin{equation}\label{E:expectation-for-zeros}
\lim_{\leb\gives \infty}\E{Z_{\leb,\,r_\leb}^x(1)}=\frac{1}{\sqrt{\pi n}}\frac{\Gamma\lr{\frac{n+1}{2}}}{\Gamma\lr{\frac{n}{2}}}.
\end{equation}
We prove \eqref{E:expectation-for-zeros} in Section \ref{S:constants} below and focus here on proving \eqref{E: equal to infty}. 
{ To do this, note that we may apply the Kac-Rice formula \eqref{E:KR} with $m=2$ to $\phi_\leb^x$ for all $\leb$ sufficiently large. Indeed, the fields $\phi_\leb^x$ are smooth for all $\leb.$ Moreover, by the $C^\infty$ convergence of covariance kernels \eqref{E:CovConv1}, we have
\[\Cov\lr{\phi_\leb^x(u),d\phi_\leb^x(u)}=\Cov\lr{\phi_\infty^x(u),d\phi_\infty^x(u)}+o_{\leb}(1).\]
Hence, since  $\Cov\lr{\phi_\infty^x(u),d\phi_\infty^x(u)}$ is non-degenerate by Proposition \ref{P:Surjective zeros}, so is the left hand side for all $\leb$ sufficiently large. Finally, again by \eqref{E:CovConv1}, there exists $\Lambda >0$ so that
\[\inf_{\leb\geq \Lambda} \inf_{u\in B_{r_\leb}}\inf_{\w\in S^{n-1}} \det \Cov\lr{\phi_\leb^x(u),\nabla_\w \phi_\leb^x(v)}>0.\]
Hence, by Proposition \ref{P:L1Density} there exists $C>0$ so that
\begin{equation}\label{E:zeros-den-est}
    \inf_{\leb\geq \Lambda} \inf_{\substack{\abs{u-v}\leq 1\\ u\in B_{r_\leb}}} \Den_{(\phi_\leb^x(u),\phi_\leb^x(v))}\leq C \abs{u-v}^{-1}.
\end{equation}
Thus, in particular, $\Den_{(\phi_\leb^x(u),\phi_\leb^x(v)})$ exists for all $u\neq v$ and all $\leb$ sufficiently large. We then have that, in the notation of the Kac-Rice formula, $\Var\big[Z_{\leb, r_\leb}^x(1)\big]$ is
\begin{align}
    \label{E:zeros-var-form}\frac{1}{\vol(B_{r_\leb})^{2}}\int_{\R^n}\int_{\R^n} {\bf 1}_{\set{B_{r_\leb}\x B_{r_\leb}}}(u,v)\Big[ &Y_{2,\phi_\leb^x}(u,v)\Den_{\phi_\leb^x(u),\phi_\leb^x(v)}(0,0) ~-\\
    \notag & \;\;Y_{1,\phi_\leb^x}{(u)}  Y_{1,\phi_\leb^x}{(v)} \Den_{\phi_\leb^x(u)}(0)\Den_{\phi_\leb^x(v)}(0)\Big]dudv.
\end{align}
To show {that \eqref{E:zeros-var-form}} converges to zero, }note that by \eqref{E:LimitDef} and \eqref{E:CovConv1}, there {exists} $\Lambda >0$ so that, uniformly over $\leb \geq \Lambda$ and  $u,v\in B_{r_\leb}$, for $\psi \in \set{\phi_\leb^x,\phi_\infty,\,d\phi_\leb^x,\,d\phi_\infty}$ we have
\begin{equation}\label{E:DeCorr}
\Cov\lr{
\psi(u) ,\psi(v)}=
\left(\begin{array}{ccc}
  \Cov\lr{\psi(u)}& 0\\
 0& \Cov\lr{\psi(v)} 
\end{array}\right)+o(1),
\end{equation}
as $\abs{u-v}\gives \infty.$ Hence, combining \eqref{E:DeCorr} with Lemma \ref{L:Gaussian-integrals}, we find
\[\sup_{u,v\in B_{r_\leb},\, \abs{u-v}\geq r_\leb^{1/2}}\abs{Y_{_{\!2,\phi_\leb^{x}}}(u,v)-Y_{_{\!1,\phi_\leb^{x}}}(u)Y_{_{\!1,\phi_\leb^{x}}}(v)} =o(1).\]
Similarly, 
\[\sup_{u,v\in B_{r_\leb},\, \abs{u-v}\geq r_\leb^{1/2}}\abs{\Den_{_{(\phi_\leb^{x}(u),\phi_\leb^x(v))}}(0,0)-\Den_{_{\phi_\leb^{x}(u)}}(0)\Den_{_{\phi_\leb^{x}(v)}}(0)} =o(1)\]
as $\leb \gives \infty.$ On the other hand, 
\[\lim_{\leb \gives \infty}\frac{\vol\set{(u,v)\in B_{r_\leb}\x B_{r_\leb}:\; {\abs{u-v}\geq r_\leb^{1/2}}}}{\vol(B_{r_\leb}\x B_{r_\leb})} =1.\]
Hence, the integrand in the expression above for $\Var[Z_{\leb,r_\leb}^x(1)]$ goes to zero pointwise on a sequence of sets whose measures tends to one. 
{
To complete the proof of \eqref{E: equal to infty} it therefore remains to show that the integrand in \eqref{E:zeros-var-form} is also uniformly dominated by an $L^1$ function. To do this, we repeat the proof of Lemma \ref{L:dom-conv} to see that $\phi_\leb$ satisfies the hypotheses of Proposition \ref{P: bound of numerator} for all $\leb$ exceeding some $\Lambda >0$. Hence, 
\[\sup_{\leb \geq \Lambda}\sup_{\substack{u\neq v\\ u\in B_{r_\leb}}}Y_{2,\phi_\leb}(u,v)<\infty. \]
Similarly, simply by appealing to \eqref{E:LimitDef} and \eqref{E:CovConv1}, we have 
\[\sup_{\leb \geq \Lambda}\sup_{u\in B_{r_\leb}}Y_{1,\phi_\leb}(u)\Den_{\phi_\leb^x(u)}(0)<\infty\]
as well. {These bounds, together with \eqref{E:zeros-den-est}, yield}  that there is a constant $C>0$ so that the integrand in \eqref{E:zeros-var-form} is bounded by 
\[\frac{1}{\vol(B_{r_\leb})^{2}}{\bf 1}_{\set{B_{r_\leb}\x B_{r_\leb}}}(u,v)\x \begin{cases}
C\abs{u-v}^{-1} &\quad \abs{u-v}\leq 1,\\
C &\quad \abs{u-v}>1.
\end{cases}\]
This completes the proof of \eqref{E:expectation} and \eqref{E: equal to infty}.}

 The proof of \eqref{E:CritDist1} and \eqref{E:CritDist2} is to repeat the preceeding argument with $\phi_\leb$ replaced by $d\phi_\leb.$ The limit in \eqref{E:CritDist2} is independent of $x$ since $\Crit_{\infty,r}^x(1)$ is independent of linear changes of coordinates on $T_xM$ and, up to such a change of coordinates, $g_x$ coincides with the Euclidean metric. This completes the proof of Theorems \ref{T:LocalZeros} and \ref{T:LocalCrits}. 

\subsection{Computation of Explicit Constants}\label{S:constants} Before proving Theorem \ref{T:Global} 
we note that the statements in Remarks \ref{R:zeros-jets1} and \ref{R:crits-jets1} follow from the extended
Kac-Rice formula (Remark \ref{R:KR Extension}). We also note that the
first moment asymptotics \eqref{E:expectation} follow from explicit computation of the limit in the Kac-Rice
formula:
\[\frac{1}{\sqrt{2\pi}}\E{\|d\phi_\infty(0)\| \,|
  \phi_\infty(0)=0}=\lr{\frac{n}{2\pi}}^{n/2}  \int_{\R^n} |\xi|
e^{-n\abs{\xi}^2/2} d\xi =\frac{1}{\sqrt{\pi n}} \frac{\Gamma
  \left(\frac{n+1}{2} \right)}{\Gamma \left(\frac{n}{2} \right)},\]
which in particular confirms \eqref{E:expectation-for-zeros}. To obtain \eqref{E:crit loc}, note that
\[\Cov(\dell_1\phi_\infty(0),\dell_2\phi_\infty(0),
\dell^2_{11}\phi_\infty(0), \dell^2_{22}\phi_\infty(0),
\dell^2_{12}\phi_\infty(0))\] 
is
\[
\begin{pmatrix}
1/2&0&0&0&0\\
0&1/2&0  &0   &0\\
0&0  &3/8&1/8&0\\
0&0&1/8&3/8&0\\
0&0&0&0&1/8
\end{pmatrix}.
\]
Hence,
\begin{align*}
&\E{|\det \text{Hess}\, \phi_\infty (0)|\,|\,  d\phi_\infty (0)=0}\\
&\qquad= \frac{8}{(2\pi)^\frac{3}{2}}  \int_{\R^3} |x_1x_2-x_3^2| e^{-\frac{1}{2}(3x_1^2+3x_2^2-2x_1x_2+8x_3^2)} dx_1dx_2dx_3\\
&\qquad = \frac{1}{8(2\pi)^\frac{3}{2}}  \iiint_{\R^3} |2y_1^2-y_2^2-y_3^2| e^{-\frac{1}{2}(y_1^2+y_2^2+y_3^2)} dy_1dy_2dy_3\\
&\qquad = \frac{1}{8(2\pi)^\frac{3}{2}} \left[\int_0^{+\infty}r^4e^{-\frac{r^2}{2}}dr \right] \left[\int_0^\pi \int_0^{2\pi}|3 \sin^2\theta\cos^2\varphi-1| \sin \theta d\varphi d\theta \right] = \frac{1}{4 \sqrt{6}}.
\end{align*}
The changes of variables we used are $y_1=x_1+x_2$,
$y_2=\sqrt{2}(x_1-x_2)$, $y_3=\sqrt{8}x_3$ and spherical coordinates
$(r, \varphi,\theta) \in (0,\infty) \times (0,2\pi) \times (0,
\pi)$. Combining this with $\Den_{d\phi_\infty(0)}(0,0)=\frac{1}{\pi}$
and the Kac-Rice formula confirms \eqref{E:crit loc}. 
%==========================================================
%==========================================================
%==========================================================
%==========================================================

\section{Global Result - Proof of Theorem \ref{T:Global}}\label{S: other cors} 
  
Let $\phi_\leb\in \MRW_\leb(M,g)$ and suppose that $M$ is a manifold
of isotropic scaling (Definition \ref{D:IS}) and that random waves on $(M,g)$ have
short-range correlations (Definition \ref{D:SRC}). We derive Theorem \ref{T:Global} 
 from Theorems \ref{T:LocalZeros} and
\ref{T:LocalCrits}, respectively. The derivation for zeros and critical points is essentially
identical, so we will focus on proving Theorem \ref{T:Global} for critical points and
will indicate the necessary changes to prove the zero sets statements as we go. 

We first prove the estimates \eqref{E: expected size} and
\eqref{E: expected critical}. It is enough to do this for indicator functions $\psi = {\bf
1}_A$ for any $A\subseteq M.$ Fix $\ep>0.$ For each $\leb$ partition $A$ into finitely many disjoint subsets $\set{ U_{\alpha}}_{\alpha \in S_\leb}$ so that 
$A=\bigcup_{\alpha \in S_\leb}U_{\alpha }$ and for some $c,C>0$
\[c \cdot \leb^{-1+\ep}\leq \text{diam}(U_{\alpha })\leq  C\cdot \leb^{-1+\ep}\] 
as $\leb \to \infty.$ For each $\alpha \in S_\leb$ choose $x_{\alpha,\leb} \in U_\alpha$ and write
\[A_{\alpha, \leb}:=\{u\in T_{x_{\alpha,\leb}}M:\;
\exp_{x_{\alpha,\leb}}(u/\leb)\in U_\alpha\cap A\}.\]
We have, 
\begin{align*}
\frac{\E{\Crit_\leb({\bf 1}_A)}}{\leb^n\vol(A)}
&=\sum_{\alpha_\leb \in S_\leb} \frac{\vol(U_{\alpha})}{\vol(A)}\E{
  \frac{\Crit_\leb({\bf 1}_{U_\alpha})}{\leb^n \vol(U_\alpha)}}\\
&= \sum_{\alpha \in S_\leb}\frac{\vol(U_{\alpha
  })}{\vol(A)} \,\E{\Crit_{\leb, \leb^\ep}^{x_{\alpha, \leb}}({\bf
  1}_{A_{\alpha, \leb}})},
\end{align*}
where $\Crit_{\leb, r_\leb}^x$ is defined in
\eqref{E:LocalCritMeas}. Combining \eqref{E:CritDist2} (see Theorem
\ref{T:LocalCrits}) with Remarks \ref{R:gensetscrits} and \ref{R:localunifcrits} and the previous line proves \eqref{E:
  expected critical}. 

We now seek to prove \eqref{E:var} and \eqref{E:var crits}. As before, it is
enough to take as our test function the indicator ${\bf 1}_A$ for $A\subseteq M$
measurable. The proofs of \eqref{E:var crits} and
\eqref{E:var} are identical, and we provide the details below for
\eqref{E:var crits}.

For each $x \in M$ write $T_{x,\leb}:=\Var(d \phi_\leb(x))=d_xd_y|_{x=y}\Pi_\leb(x,y)$. Proposition
\ref{P:Surjective crits} and the discussion in Section  \ref{S:KR crits} ensure that $T_{\leb,x}$ is an invertible matrix at every $x$ for all $\leb$ sufficiently large. We may therefore set
\[\psi_{ \leb}(x):= T_{x,\leb}^{-\frac{1}{2}} \, d\phi_\leb(x),\]
which yields
\[ \Var(\psi_{ \leb}(x))= \text{Id},\quad \qquad \forall \; x \in M.\]
Note that $\psi_{ \leb}$ and $d\phi_\leb$ have the same zero set. It will turn out to be more convenient to study the variance for the size of the zero set of  $ \psi_{ \leb}$. Let us write 
\[X_{\leb}:=\leb^{-n}\#\{A\cap\psi_\leb^{-1}(0)\}.\] 
{ {We first} check that we can apply the Kac-Rice formula to studying the first two moments of $X_\leb.$ Since $X_\leb$ are smooth, we need only check that for $\leb$ sufficiently large {the Gaussian vector} $\lr{\psi_\leb(x),\psi_\leb(y)}$ is non-degenerate for all $x\neq y.$ Note that for all $\leb$ sufficiently large, $\psi_\leb(x)$ is uniformly non-degenerate for all $x\in M$ by the isotropic scaling assumption \eqref{D:IS}. Thus, when $x,y\in M$ satisfy $d_g(x,y)>\leb^{-\ep},\,\ep\in (0,1),$ the short range correlation assumptions \eqref{D:SRC} and \eqref{E:DeCorr} immediately show that $\lr{\psi_\leb(x),\psi_\leb(y)}$ is non-degenerate for all $\lambda$ large enough. In contrast, if $d_g(x,y)<\leb^{-\ep}$, {we} observe that $\Cov\lr{\psi_\leb(x),\psi_\leb(y)}$ is
\begin{equation}\label{E:psi-cov}
   P \lr{\begin{array}{cc}
    d_xd_y|_{y=x}\Pi_\leb(x,y) & d_xd_y\Pi_\leb(x,y) \\
    d_xd_y\Pi_\leb(y,x) & d_xd_y|_{x=y}\Pi_\leb(x,y)
\end{array}}P,\qquad  P:=\lr{\begin{array}{cc}
    T_{x,\leb}^{-1/2} &0  \\
     0&T_{y,\leb}^{-1/2} 
\end{array}}.
\end{equation}
The middle matrix is precisely $\Cov\lr{d\phi_\leb(x),d\phi_\leb(y)}.$ We already saw in Lemma \ref{L:KR-crits} that this matrix is invertible for every $x\neq y$ with $d_g(x,y)<\leb^{-\ep}.$ 
{Also note that}
$T_{x,\leb}=\Cov\lr{d\phi_\leb(x)}=\leb^2\Cov\lr{d\phi_\leb^x(0)}=\leb^2\lr{\Cov\lr{d\phi_\infty(0)}  +o(1)}$
uniformly for all $x\in M$. {In particular,} 
\begin{equation}\label{E:T-est}
    T_{x,\leb}~=~ \leb^{2}\Big(-\tfrac{1}{n}\mathrm{Id}+o_\leb(1)\Big)\qquad \text{and}\qquad  T_{x,\leb}^{-1/2}~=~ \leb^{-1}\Big(-n^{1/2}\mathrm{Id}+o_\leb(1)\Big).
\end{equation} 
{It then follows that} $\Cov\lr{\psi_\leb(x),\psi_\leb(y)}$ is also invertible for all $x\neq y$. This confirms that {we} can apply the Kac-Rice formula} to write
\begin{equation}
\Var[X_{\leb}]=\iint_{A \x A} \lr{I_{_{\!2,\psi_\leb}}(x,y)-I_{_{\!1,\psi_\leb}}(x)I_{_{\!1,\psi_\leb}}(y)}dv_g(x)dv_g(y), \label{E:VarInt}
\end{equation}
where 
\[I_{_{\!2,\psi_\leb}}(x,y)=Y_{_{\!2, \psi_\leb/\leb}}(x,y)\cdot
\text{Den}_{\psi_\leb(x)/\leb,\psi_\leb(y)/\leb}(0,0)\]
and
\[I_{_{\!1,\psi_\leb}}(x)=Y_{_{\!1,\psi_\leb/\leb}}(x) \cdot \text{Den}_{\psi_\leb(x)/\leb}(0),\] 
with 
\begin{align*}
&Y_{_{\!1,\psi_\leb/\leb}}(x)=\E{\big\|\,D\lr{\leb^{-1}\psi_\leb(x)} \big\| \; \big | \; \psi_\leb(x)=0},\\
&Y_{_{\!2,\psi_\leb/\leb}}(x,y)=\E{\big\|\,D\lr{\leb^{-1}\psi_\leb(x)} \big\| \big\| \,D\lr{\leb^{-1}\psi_\leb(y)} \big\|\;\big |\; \psi_\leb(x)=\psi_\leb(y)=0},
\end{align*}
where we have abbreviated
$\norm{Df(x)}=\left[\det\lr{df(x)^*df(x)}\right]^{1/2}.$ We will decompose the integral in \eqref{E:VarInt} into three $\leb$-dependent pieces using the following construction. There exist three positive numbers $C_1, C_2, C_3$ depending only on $n=\dim(M)$ with the following properties. For each $\ep>0$ and every $\leb>0$ there exists a collection of measurable sets $\set{B_\alpha}_{\alpha\in S_{\leb,\ep}}$ satisfying
\begin{enumerate}
\item[i)\,\,\,] $\#S_{\leb,\ep} \leq C_1\leb^{n\ep}$.
\vskip .1cm
\item[ii)\,\,] $\diam(B_\alpha)\leq \leb^{-\ep}$ for every $\alpha \in S_{\ep, \leb}$.
\vskip .1cm
\item[iii)] $\set{x,y\in M:\, d_g(x,y)<\leb^{-2\ep}} \subset \; \Union_{\alpha \in S_{\leb,\ep}} B_{\alpha}\x B_{\alpha}$.
\vskip .1cm
\item[iv)\,] For any $K>C_2$ and distinct $\alpha_1,\ldots, \alpha_K\in
  S_{\leb,\ep}$ we have $\bigcap_{i=1}^K B_{\alpha_i}=\emptyset.$
\vskip .1cm
\item[v)\,\,\,]For each $B_\alpha\in S_{\leb, \ep}$ and every $2\leq k \leq C_2$ 
\vskip -.2cm\[\#\left\{\mathrm{distinct}~~\alpha_2,\ldots, \alpha_k\in
  S_{\leb,\ep}~|~ B_\alpha \cap \bigcap_{i=2}^k B_{\alpha_i}\neq \emptyset \right\} \leq C_3.\]
\end{enumerate}
To see this, cover $M$ with finitely many coordinate charts. On each chart $g$ is uniformly comparable to the Euclidean metric. For the Euclidean metric, the existence of a collection satisfying (i)-(v) follows from standard covering arguments. Taking the union of these collections over the finite number of coordinate charts completes the construction of the sets  $\set{B_\alpha}_{\alpha\in S_{\leb,\ep}}$ satisfying (i)-(v). Setting $\ep:=\frac{n-1}{2n}$, write
\begin{equation}\label{E:decomposition}
  \Var[X_{\leb}] = \sum_{j=1}^3\underbrace{\iint_{\Omega_{j,\leb}\cap
      A\x A} \lr{I_{_{\!2,\psi_\leb}}(x,y)-I_{_{\!1,\psi_\leb}}(x)I_{_{\!1,\psi_\leb}}(y)}dv_g(x)dv_g(y)}_{=:W_{j,\leb}},
\end{equation}
where 
\begin{align*}
\Omega_{1,\leb}&=\bigcup_{\alpha \in S_{\leb, \ep}} B_{\alpha}\x
                 B_{\alpha},\quad \Omega_{2,\leb}=\Omega_{1,\leb}^c
                 \cap V_{\leb},\quad \Omega_{3,\leb}=\Omega_{1,\leb}^c\cap V_{\leb}^c,
\end{align*}
and 
\begin{equation}\label{E:V}
V_{ \leb}=\left \{(x,y) \in M \times M :\; \max_{\alpha, \beta \in\set{0,1}}\set{ \leb^{-\alpha-\beta}\,|\nabla_x^\alpha \nabla_y^\beta\, \Cov(\psi_\leb(x),\psi_\leb(y))|}>\leb^{-\frac{n-1}{4}}\right \}.
\end{equation}
The proof of \eqref{E:var crits} now reduces to the following three estimates:
\begin{equation}\label{E:Goal}
  W_{i,\leb}=O(\leb^{-\frac{n-1}{2}}),\qquad \text{as}\quad \leb \gives \infty,\quad i=1,2,3.
\end{equation}
We begin by proving \eqref{E:Goal} for $i=1$. Consider any $B\subseteq M$ with $\diam(B)\leq \inj(M,g)$ and fix $x\in B$. Write 
\[B_{x,\leb}:=\{u\in T_xM:\;\exp_x(u/\leb)\in B\},\]
and note that for each $B\subset M$ 
\[\frac{\#\lr{\{\psi_\leb^x = 0\}
    \cap
    B_{x,\leb}}}{\vol(B_{x,\leb})}=\frac{\#\lr{\psi_\leb^{-1}(0)
    \cap B}}{\leb^n \vol(B)},\] 
where $\psi_\leb^x(u)=\psi_\leb(\exp_x(u/\leb)).$ By the Kac-Rice formula,
\begin{align}
\frac{1}{ \vol(B)^2}\notag&\iint_{B\x B}\lr{I_{2, \psi_\leb}(x,y)-I_{_{\!1,\psi_\leb}}(x)I_{_{\!1,\psi_\leb}}(y)}dv_g(x)dv_g(y)\\
\label{E:singleset}&\qquad= \Var \lr{ \frac{\#\lr{\psi_\leb^{-1}(0)
    \cap B}}{\leb^n \vol(B)}}=  \Var\lr{\frac{\#\lr{\{\psi_\leb^x = 0\}
    \cap
    B_{x,\leb}}}{\vol(B_{x,\leb})}}.
\end{align}
Since $(M,g)$ is a manifold of isotropic scaling, this last
expression is uniformly bounded over $x\in M$ (see Remark \ref{R:localunifcrits}). Using the inclusion-exclusion formula and property (iv)
of $S_{\leb, \ep}$, we have the following decomposition for the
indicator function of $\Omega_{1,\leb}:$
\[{\bf 1}_{\Omega_{1,\leb}}=\sum_{j=1}^{C_2} (-1)^{j+1} \sum_{\substack{\mathrm{distinct}\,\,\alpha_i \in S_{\leb, \ep}\\ i=1,\ldots, j}} {\bf 1}_{B_{\alpha_{1,\ldots,j}}\x B_{\alpha_{1,\ldots, j}}},\]
where $B_{\alpha_{1,\ldots,j}}:=B_{\alpha_1}\cap \cdots \cap
B_{\alpha_j}.$ By properties (i) and (v), for each $j,$ the number of
terms in the inner sum is at most $C_1\,C_3\, \leb^{n\ep}.$ Note that
by (ii), we have $\vol(B)\leq \leb^{-n\ep}$ for each $B\in
\set{B_\alpha}_{\alpha\in S_{\leb, \ep}}.$ For each $\alpha \in
S_{\leb,\ep},$ choose $x_{\alpha,\leb}\in B_\alpha.$ Relation
\eqref{E:CritDist2} in Theorem
\ref{T:LocalCrits} together with Remark \ref{R:localunifcrits} and
\eqref{E:singleset} shows
\begin{equation}\label{E:Varest}
  \sup~~\left\{\Var
    \Big[\frac{\#\lr{\psi_\leb^{-1}(0)
    \cap B}}{\leb^n \vol(B)} \Big]~:~ B\neq \emptyset \text{ finite intersection of sets in } \set{B_\alpha}_{\alpha \in S_{\leb, \ep}}\right\} =O(1)
\end{equation}
as $\leb\gives \infty.$ Combining this with \eqref{E:singleset}, we find
\begin{align*}
  W_{1,\leb} &= \sum_{j=1}^{C_2} (-1)^{j+1}
  \sum_{\substack{\mathrm{distinct}\,\,\alpha_i \in S_{\leb, \ep}\\
  i=1,\ldots, j}} \vol(B_{\alpha_{1,\ldots,j}})^2\Var\left[\frac{\#\lr{\psi_\leb^{-1}(0)
    \cap B_{\alpha_{1,\ldots,j}}}}{\leb^n \vol(B_{\alpha_{1,\ldots,j}})}\right]=O\lr{\leb^{-n\ep}},
\end{align*}
which confirms \eqref{E:Goal} for $i=1$ since $\ep=(n-1)/(2n).$ Next, to prove \eqref{E:Goal} for $i=2,3$ we will need the following estimate:
\begin{equation}\label{E: L2bound}
\left\|  \leb^{-|\alpha|-|\beta|}\,\nabla_x^\alpha \nabla_y^\beta\, \big(\Cov(\psi_\leb(x),\, \psi_\leb(y) \big)_{i,j} \right\|^2_{L^2(M \times M)} =O_{\alpha, \beta}(\leb^{-n +1})
\end{equation}
as $\leb\gives \infty$ for all $\alpha, \beta \geq 0,\,\, 1\leq i, j
\leq n.$ We postpone the proof of \eqref{E: L2bound} until the end of
this section. Assuming it for the moment, abbreviate
\[\Cov(\psi_\leb(x),\, \psi_\leb(y))=\Sigma_\leb(x,y)= \left(\begin{array}{cc}\Sigma_{x,x} & \Sigma_{x,y} \\ \Sigma_{x,y}^T&  \Sigma_{y,y}\end{array}\right).\]
By construction 
\begin{equation}\label{E:nicecov}
\Sigma_{x,x}=\Sigma_{y,y}=\text{Id},\qquad d_x|_{x=y} \Sigma_{x,y}=0.
\end{equation}
Combining Chebyshev's inequality with the definition of $V_{\leb}$
with \eqref{E: L2bound} yields
 \begin{equation}\label{E:volume}
    \vol_g(V_{ \leb})=O(\leb^{-\frac{n-1}{2}}).
  \end{equation}
  Next, we claim that
\begin{equation}\label{E:denest}
\sup_{\substack {x,y\in V_{\leb} \\ d_g(x,y)>\leb^{-2\ep}}} \Big |I_{2, \psi_\leb}(x,y)-I_{_{\!1,\psi_\leb}}(x)I_{_{\!1,\psi_\leb}}(y) \Big| =O(1)
\end{equation}
as $\leb\gives \infty.$ Indeed, the Definition \ref{D:SRC} of
short-range correlations ensures that the density factor
$\Den_{\psi_\leb(x),\psi_\leb(y)}(0,0)$ is uniformly bounded above on
$V_{\leb}\cap \set{d_g(x,y)>\leb^{-2\ep}}.$ 
{
Combining Lemma \ref{L:Gaussian-integrals} with Lemma \ref{L:M-non-degen} below, we find that
$Y_{_{\!2,\psi_\leb}}(x,y)$ and $Y_{_{\!1,\psi_\leb}}(x)Y_{_{\!1,\psi_\leb}}(y)$ are uniformly bounded on $V_{\leb}\cap \set{d_g(x,y)>\leb^{-2\ep}}.$
Let
\begin{equation}\label{E:condcov}
M_\leb(x,y)=\Cov
\big(\leb^{-1}{d\psi_\leb}(x),\leb^{-1}d\psi_\leb({y})~|~{\psi_\leb}(x)=
\psi_\leb(y) =0 \big).
\end{equation}

\begin{Lem}\label{L:M-non-degen}
There exists $C>0$ such that 
\begin{equation}\label{E:M-est}
\sup_{x\neq y} M_\leb(x,y)\leq C,
\end{equation}
where the inequality applies to each entry of $M_\leb(x,y).$ 
\end{Lem}
\begin{proof}
The idea of the proof is that the estimate \eqref{E:M-est} in the far off-diagonal regime (when $d_g(x,y)>\leb^{-\ep}$) essentially follows immediately from the pointwise Weyl Law and the short range correlation assumption \eqref{D:SRC}. In the near off-diagonal regime $0<d_g(x,y)<\leb^{-\ep}$ the estimate \eqref{E:M-est} follows closely the proof of Lemma \ref{L:cov-deriv-bounds}, which gives essentially the same result but for the local fields $d\phi_\leb^{x_0}$. But, since $\psi_\leb(x) = T_{x,\leb}^{-1/2} d\phi_\leb(x)$ and the matrices $T_{x,\leb}\approx \leb^{-2}\mathrm{Id}$ (see \eqref{E:T-est}), the two fields have the same local behavior. We now provide the details. Note that
\begin{equation}\label{E:dpsi-cond-cov}
M_\leb(x,y)=
\leb^{-2}\Cov\lr{d\psi_\leb(x),d\psi_\leb(y)}-\leb^{-2}\!A(x,y)^T\!\Cov\lr{\psi_\leb(x),\psi_\leb(y)}^{-1}\!\!A(x,y),\end{equation}
where 
\[
A=\lr{\begin{array}{cc}
    \Cov\lr{d\psi_\leb(x), \psi_\leb(x)} & \Cov\lr{d\psi_\leb(x), \psi_\leb(y)} \\
     \Cov\lr{d\psi_\leb(y), \psi_\leb(x)}& \Cov\lr{d\psi_\leb(y), \psi_\leb(y)}
\end{array}},\]
and we will see below that for all $x\neq y$ the matrix $\Cov\lr{\psi_\leb(x),\psi_\leb(y)}$ is invertible. Since $\Cov(AX,BY)=A^T\Cov(X,Y)B,$ the covariance matrix $\Cov\lr{d\psi_\leb(x),d\psi_\leb(y)}$ is
\[\lr{\begin{array}{cc} C(x,x) & C(x,y)\\ C(y,x) &C(y,y)\end{array}},\qquad\qquad C(x,y)=d_xd_y\lr{T_{x,\leb}^{-1/2}d_xd_y\Pi_\leb(x,y)T_{y,\leb}^{-1/2}}.\]
Recall that for any $\ep\in (0,1),$ the short range correlations assumption \eqref{D:SRC} gives that for any multi-indices $\alpha, \beta$
 \begin{equation}\label{E:Pi-est}
     \sup_{d_g(x,y)>\leb^{-\ep}}|d_x^\alpha d_y^\beta \Pi_\leb(x,y)| ~=~ O(\leb^{\abs{\alpha}+\abs{\beta}});
 \end{equation}
 and, moreover, if $\abs{\alpha}+\abs{\beta}$ is odd, then
 \begin{equation}\label{E:Pi-est2}
    |d_x^\alpha d_y^\beta|_{x=y} \Pi_\leb(x,y)|~=~ o(\leb^{\abs{\alpha}+\abs{\beta}}).
 \end{equation}
Hence, $\leb^{-2}\Cov\lr{d\psi_\leb(x),d\psi_\leb(y)}$ is uniformly bounded for all $x,y\in M$, and it remains to check that the second term in \eqref{E:dpsi-cond-cov} is also uniformly bounded  for all $x\neq y.$ We do this, by considering separately the near off-diagonal and far off-diagonal regimes separately.
We start with far off-diagonal regime in which we fix $\ep>0$ and consider  points $x,y$ with $d_g(x,y)>\leb^{-\ep}.$ Combining \eqref{E:psi-cov} with \eqref{E:T-est} yields
\[\Cov\lr{\psi_\leb(x),\psi_\leb(y)}^{-1} = \mathrm{Id}+o_\leb(1).\]
Hence, there exists $C>0$ such that for all $\leb$ sufficiently large
\[\sup_{d_g(x,y)>\leb^{-\ep}}M_\leb(x,y)\leq C.\]
To handle the near off-diagonal regime, where $0<d_g(x,y)<\leb^{-\ep}$, let us write in local coordinates
\[x =x_0+\frac{u}{\leb},\qquad y=x_0+\frac{v}{\leb},\qquad \psi_\leb^{x_0}(u):=\psi_\leb\lr{x_0+\frac{u}{\leb}}.\]
Thus, with $X,Y,W,Z$ as in  Lemma \ref{L:cov-deriv-bounds}, the expression \eqref{E:psi-cov} becomes
\[\Cov\lr{\psi_\leb(x),\psi_\leb(y)}^{-1}=\lr{\begin{array}{cc}
    \leb^{-1} T_{x,\leb}^{1/2} &0  \\
     0&\leb^{-1} T_{y,\leb}^{1/2} 
\end{array}} \lr{\begin{array}{cc}
   X & Y \\
   W & Z
\end{array}}\lr{\begin{array}{cc}
  \leb^{-1}  T_{x,\leb}^{1/2} &0  \\
     0&\leb^{-1} T_{y,\leb}^{1/2} 
\end{array}},\]
where the extra factors of $\leb$ come from the fact that  $\leb^{-1}d\psi_\leb(x)=d\psi_\leb^{x_0}(u).$ Thus, we find that the second term in \eqref{E:dpsi-cond-cov} equals
\begin{equation}\label{E:2nd-term-psi}
B^T \lr{\begin{array}{cc}
   X & Y \\
   W & Z
\end{array}}B\qquad \text{with}\qquad  B=\lr{\begin{array}{cc} M\leb^{-1}T_{x,\leb}^{1/2} & \lr{M+\ep}\leb^{-1}T_{y,\leb}^{1/2}\\ \lr{M+\ep}\leb^{-1}T_{x,\leb}^{1/2} &\lr{M +\ep}\leb^{-1}T_{y,\leb}^{1/2}\end{array}},
\end{equation}
where $M=M(x)$ and $\ep=\ep(x,y)$ are  matrices with 
$M=\E{\leb^{-1}d_x \psi_\leb(x) \psi_\leb(x)}$ and $\ep= O(\abs{x-y})$.
The $2\x 2$ block matrix resulting from multiplying the three terms in \eqref{E:2nd-term-psi} is a function only of $X+Y+W+Z$, $\ep^T\lr{Y+Z}$, $\lr{W+Z}\ep$ and $\ep^T Z\ep $ times the matrices $\leb^{-1} T_{x,\leb}^{1/2}$. The estimates {in} \eqref{E:T-est} combined with {the} bounds just below \eqref{E:Rij-goal} show that each such term is uniformly bounded, completing the proof of Lemma \ref{L:M-non-degen}. \\
\end{proof}
}

\noindent Combining \eqref{E:volume} with \eqref{E:denest}, we have
\begin{equation}\label{E:A2}
\int_{\Omega_{2,\leb}}\lr{I_{2, \psi_\leb}(x,y)-I_{_{\!1,\psi_\leb}}(x)I_{_{\!1,\psi_\leb}}(y)}dv_g(x)dv_g(y)=O(\vol_g(V_{\leb}))=O(\leb^{-\frac{n-1}{2}}),
\end{equation}
confirming \eqref{E:Goal} for $i=2.$ Finally, we check \eqref{E:Goal}
for $i=3$. By definition of $V_\leb,$ 
\[\Sigma_\leb(x,y)=\text{Id}+\;\begin{psmallmatrix}
0 \;&\; O(\leb^{-\frac{n-1}{4}})\\ O(\leb^{-\frac{n-1}{4}}) \;&\; 0
\end{psmallmatrix},\]
with the error terms uniform over $x,y\in V_\leb^c.$ Hence,
\begin{align*}
\Den_{\psi_\leb(x),\psi_\leb(y)}(0,0)&=
                                       \det(2\pi \Sigma_\leb(x,y))^{-1/2}=\Den_{\psi_\leb(x)}(0)\Den_{\psi_\leb(y)}(0)+O(\leb^{-\frac{n-1}{2}}).
\end{align*}
Thus,
\[W_{3,\leb}=\int_{\W_{3,\leb}} \lr{Y_{_{\!2,\psi_\leb}}(x,y)-Y_{_{\!1,\psi_\leb}}(x)Y_{_{\!1,\psi_\leb}}(y)}\;dv_g(x)dv_g(y)+O(\leb^{-\frac{n-1}{2}}),\]
and proving \eqref{E:Goal} for $i=3$ reduces to showing 
\begin{equation} \label{E:Goal3}
\sup_{(x,y) \in \Omega_{3,\leb}} |Y_{2, \psi_\leb}(x,y)-Y_{_{\!1,\psi_\leb}}(x)Y_{_{\!1,\psi_\leb}}(y)|=O(\leb^{-\frac{n-1}{2}}).
\end{equation}
Write
\[\Cov \big(\leb^{-1}\psi_\leb(x), \leb^{-1}\psi_\leb(y),
\leb^{-1}d\psi_\leb(x), \leb^{-1}d\psi_\leb(y) \big)= \leb^{-2}
\left(\begin{array}{cc}\Sigma_\leb(x,y) & \leb B_\leb(x,y) \\\leb
        B_\leb(x,y)^T&  \leb^2 C_\leb(x,y)\end{array}\right),\]
where by the definition of $V_{\leb}$ we have
\begin{align*}
&B_\leb(x,y)= \left(\begin{array}{cc}0 & \leb^{-1} d_x \Sigma_\leb(x,y) \\ \leb^{-1}d_y \Sigma_\leb(x,y)&  0\end{array}\right)= \left(\begin{array}{cc}0 & O(\leb^{-\frac{n-1}{4}})\\ O(\leb^{-\frac{n-1}{4}})&  0\end{array}\right)\\
&C_\leb(x,y)\!\!= \!\!\left(\begin{array}{cc}
                      \leb^{-2}d_xd_y|_{x=y}\Sigma_\leb(x,y) &
                                                              \! \leb^{-2}d_xd_y\Sigma_\leb(x,y) \\ \leb^{-2}d_xd_y\Sigma_\leb(x,y)&  \!\leb^{-2}d_xd_y|_{x=y}\Sigma_\leb(x,y)\end{array}\right)
                                                               \!\!=\!\!\left(\begin{array}{cc} C_\leb(x) &\! \!O(\leb^{-\frac{n-1}{4}}) \\ \! \!O(\leb^{-\frac{n-1}{4}})&  C_\leb(y)\end{array}\right),
\end{align*}
where the error terms are uniform in $(x,y) \in \Omega_{3,\leb}$ and
there exists $C=C(\dim(M))>0$ so that
\[\sup_{x\in M}\abs{C_\leb(x) - C\cdot \text{Id}}=o(1)\] 
as $\leb\gives \infty$ by the poinwise Weyl law. Hence, with
$M_\leb(x,y)$ defined as in \eqref{E:condcov}
\begin{align*}
M_{\leb}(x,y)&=\left(\begin{array}{cc} C_\leb(x)& 0 \\ 0 & C_\leb(y)\end{array}\right)\left(\text{Id}+\;\begin{psmallmatrix}
O(\leb^{-\frac{n-1}{2}}) \;&\; O(\leb^{-\frac{n-1}{4}})\\ O(\leb^{-\frac{n-1}{4}}) \;&\; O(\leb^{-\frac{n-1}{2}})
\end{psmallmatrix} \;\right). 
\end{align*}
In particular, 
\begin{equation}
\det(M_\leb(x,y))=\det(C_\leb(x)) \det(C_\leb(y))\lr{1+O(\leb^{-\frac{n-1}{2}})}\label{E:detest}
\end{equation}
and
\begin{align*}
&(M_{\leb}(x,y))^{-1}= \left(\begin{array}{cc} C_\leb(x)^{-1} & 0 \\ 0 & C_\leb(y)^{-1} \end{array}\right) +\left(\begin{array}{cc} Q_{\leb}(x,y)&R_\leb(x,y)  \\  R_\leb(y,x) &Q_\leb(x,y) \end{array}\right),
\end{align*}
where $R_\leb(x,y)$, $Q_\leb(x,y)$ are matrices that satisfy the entrywise estimates
\[\sup_{(x,y) \in \Omega_{3,\leb}}R_\leb(x,y)=O(\leb^{-\frac{n-1}{4}})\quad \text{and} \quad \sup_{(x,y) \in \Omega_{3,\leb}} Q_\leb(x,y)=O(\leb^{-\frac{n-1}{2}}).\]
Then, since $d\psi_\leb(x)$ is uncorrelated from $\psi_\leb(x)$ at
each $x\in M,$ we find using \eqref{E:detest} that $Y_{2, \psi_\leb}(x,y)-Y_{_{\!1,\psi_\leb}}(x)Y_{_{\!1,\psi_\leb}}(y)$ is
\begin{align*} 
&\iint_{\R^{2n}} \frac{ |\xi|\, |\zeta|\; e^{-\frac{1}{2} \langle C_\leb(x)^{-1} \xi , \xi \rangle-\frac{1}{2} \langle C_\leb(y)^{-1} \zeta , \zeta \rangle}}{(2\pi)^n (\det  C_\leb(x)\, \det C_\leb(y))^{\frac{1}{2}}  } \!\!\left(e^{-\langle R_\leb(x,y) \, \zeta , \xi \rangle -\frac{1}{2} \left\langle Q_\leb(x,y) \,  \begin{psmallmatrix} \xi \\         \zeta   \end{psmallmatrix} , \begin{psmallmatrix} \xi \\         \zeta   \end{psmallmatrix}  \right\rangle  }-1\right)\!\! d \xi d\zeta
\end{align*}
plus $O(\leb^{-\frac{n-1}{2}})$ where the implied constant is uniform over $(x,y)\in \Omega_{3,\leb}$. Observe that
\[e^{-\langle R_\leb(x,y) \, \zeta , \xi \rangle -\frac{1}{2} \left\langle Q_\leb(x,y) \,  \begin{psmallmatrix} \xi \\         \zeta   \end{psmallmatrix} , \begin{psmallmatrix} \xi \\         \zeta   \end{psmallmatrix}  \right\rangle  }=1 -\langle R_\leb(x,y) \, \zeta , \xi \rangle + O(\leb^{-\frac{n-1}{2}}) \]
and that the integral of $\langle R_\leb(x,y) \, \eta , \xi \rangle$
against the Gaussian density above is $0$. This proves \eqref{E:Goal3}
for $i=3$, and completes the proof of Theorem \ref{T:Global} modulo the proof of \eqref{E: L2bound}, which we now supply. 

\subsubsection*{Proof of \eqref{E: L2bound}.} \label{S:L2boundpf}
By \cite[Thm. 2]{CH-derivatives}, as $\leb \gives \infty$ we have for each $\gamma, \delta = 1,\ldots, n,$
\[\nabla_x^\alpha\lr{T_{x,\leb}}_{\gamma, \delta}=
C_{\alpha}\leb^{\alpha +2 }(1+o_\alpha(1)).\]
Thus, the entries of $\nabla_x^\alpha T_{x,\leb}^{-1/2}$ are bounded by a
constant times $\leb^{\alpha - 1}.$ Since
\[\Cov(\psi_\leb(x), \psi_\leb(y))=T_{x,\leb}^{-1/2} \Cov\lr{d\phi_\leb(x),\, d\phi_\leb(y)}T_{y,\leb}^{-1/2},\]
\eqref{E: L2bound} is equivalent to showing that as $\leb\gives \infty$
\begin{equation}
\left\|  \leb^{-|\alpha|-|\beta|}\,\nabla_x^\alpha \nabla_y^\beta\, \Cov(\leb^{-1}d\phi_\leb(x),\, \leb^{-1}d\phi_\leb(y))_{\gamma,\delta} \right\|^2_{L^2(M \times M)} =O_{\alpha, \beta}(\leb^{-n +1}),
\end{equation}
We will show more generally that if $P=\mathrm{Op}(p), Q=\mathrm{Op}(q) $ are
pseudodifferential operators with orders $\text{ord}(P), \, \text{ord}(Q)$ acting on $M\times M$, then
\begin{equation}\label{E:fingoal}
\langle P\Pi_{\leb} \, , \, Q\Pi_{\leb} \rangle_{L^2(M \times M)}= O(\leb^{-n+1+\text{ord}(P) + \text{ord}(Q)}), 
\end{equation}
as $\leb \to \infty$, where the implied constant is uniform when $\norm{p}_{L^2}, \norm{q}_{L^2}$ are bounded. We prove this by induction in $\text{ord} P + \text{ord} Q$. The base case is immediate since
\begin{equation}\notag
\langle \Pi_{\leb} \, , \, \Pi_{\leb} \rangle_{L^2(M \times M)}=\int_M
\frac{\Pi_{\leb}(x,x)}{{ \dim H_{\leb}}} dv_g(x) = \lr{\dim H_{\leb}}^{-1}= O(\leb^{-n+1}). 
\end{equation}
Assume \eqref{E:fingoal} is true for all operators whose orders sum to at most $\ell-1$ and consider $P,Q$  with  $\text{ord} P+\text{ord} Q=\ell$.  Then, 
\begin{align*}
\langle P\Pi_{\leb} \, , \, Q\Pi_{\leb} \rangle
&= \left\langle    P \Delta_x^{{1}/{2}} \Pi_{\leb} \, , \, \Delta_x^{-{1}/{2}} Q\Pi_{\leb} \right\rangle+ \left\langle    [ \Delta_x^{{1}/{2}}, P] \Pi_{\leb} \, , \, \Delta_x^{-{1}/{2}} Q\Pi_{\leb} \right\rangle.
\end{align*}
Note that $\text{ord}([ \Delta_x^{{1}/{2}}, P] )+ \text{ord}( \Delta_x^{-{1}/{2}} Q)\leq \gamma-1$. Hence, the second term is $O(\leb^{-n+\gamma}).$ Finally, $\Delta_x^{{1}/{2}} \Pi_{\leb} =\leb \Pi_{\leb} + R_\leb \Pi_{\leb}$ where $R_\leb=\mathrm{Op}(r_\leb)$ is order $0$ pseudodifferential operator with $\norm{r_\leb}_{L^2}$ uniformly bounded in $\leb$. Therefore, $\left\langle    P \Delta_x^{{1}/{2}}  \Pi_{\leb} \, , \, \Delta_x^{-{1}/{2}} Q\Pi_{\leb} \right\rangle=O(\leb^{-n+1+\ell} ),$ concluding the proof of \eqref{E:fingoal}.

\end{document}